\documentclass[12pt,oneside,reqno]{amsart}
\usepackage[margin=1in]{geometry}
\usepackage{color}
\usepackage{esint,amssymb}
\usepackage{graphicx}
\usepackage{MnSymbol}
\usepackage{mathtools}
\usepackage[colorlinks=true, pdfstartview=FitV, linkcolor=blue, citecolor=blue, urlcolor=blue,pagebackref=false]{hyperref}
\usepackage{microtype}
\usepackage{amsmath}
\usepackage{xifthen}
\usepackage{verbatim}
\definecolor{darkgreen}{rgb}{0,0.5,0}
\definecolor{darkblue}{rgb}{0,0,0.7}
\definecolor{darkred}{rgb}{0.9,0.1,0.1}

\newtheorem{theorem}{Theorem}
\newtheorem{proposition}[theorem]{Proposition}
\newtheorem{lemma}[theorem]{Lemma}

\theoremstyle{definition}
\newtheorem{remark}[theorem]{Remark}

\newcommand{\ssref}[1]{Subsection~\ref{ss.#1}}
\newcommand{\tref}[1]{Theorem~\ref{t.#1}}
\newcommand{\pref}[1]{Proposition~\ref{p.#1}}
\newcommand{\lref}[1]{Lemma~\ref{l.#1}}
\newcommand{\cref}[1]{Corollary~\ref{c.#1}}

\newcommand{\sref}[1]{Section~\ref{s.#1}}

\newcommand{\eref}[1]{(\ref{e.#1})}

\numberwithin{equation}{section}
\numberwithin{theorem}{section}

\newcommand{\N}{\mathbb{N}}
\newcommand{\R}{\mathbb{R}}

\newcommand{\E}{\mathbb{E}}

\renewcommand{\a}{\bar{a}_{ij}}
\newcommand{\cm}{\bar{c}}

\newcommand{\h}{\mathfrak{H}}
\renewcommand{\H}{Y_T}
\renewcommand{\E}{E_T}

\newcommand{\eps}{\varepsilon}

\newcommand{\test}[1][]{%
\ifthenelse{\equal{#1}{}}{omitted}{given}%
}

\newcommand{\derv}[2]{\partial_x^{\alpha{#1}}\partial_v^{\beta{#2}}}
\newcommand{\der}{\derv{}{}}

\renewcommand{\d}[1]{\ensuremath{\operatorname{d}\!{#1}}}
\DeclarePairedDelimiter{\norm}{\lVert}{\rVert}

\DeclareMathOperator{\ini}{in}
\newcommand{\jap}[1]{\left\langle {#1} \right\rangle}

\renewcommand{\bar}{\overline}
\renewcommand{\tilde}{\widetilde}

\renewcommand{\part}{\partial}

\usepackage{dsfont}

\begin{document}
\title[Local existence for the Landau Equation with hard potentials]{Local existence for the Landau Equation\break with hard potentials}
\begin{abstract}
We consider the spatially inhomogeneous Landau equation with hard potentials (i.e. with $\gamma \in [0,1]$) on the whole space $\R^3$. We prove existence and uniquenss of solutions for a small time, assuming that the initial data is in a weighted tenth-order Sobolev space and has exponential decay in the velocity variable. 

In constrast to the soft potential case, local existence for the hard potentials case has been missing from the literature. This is because the moment loss issue is the most severe for these potentials. To get over this issue, our proof relies on a weighted hierarchy of norms that depends on the number of spatial and velocity derivatives in an asymmetric way. This hierarchy lets us take care of the terms that are affected by the moment loss issue the most. These terms do not give in to methods applied to study existence of solutions to Landau equation with soft potentials  and are a major reason why the local existence problem was not known for the case of hard potentials.
 \end{abstract}

\author[S. Chaturvedi]{Sanchit Chaturvedi}
\address[Sanchit Chaturvedi]{450 Jane Stanford Way, Bldg 380, Stanford, CA 94305}
\email{sanchat@stanford.edu}
\subjclass[2010]{}
\date{\today}

\maketitle
\section{Introduction}
We study the Landau equation for the particle density $f(t,x,v)\geq 0$ in the whole space $\R^3$. $t\in \R_{\geq 0}, x\in \R^3$ and $v \in \R^3$. The Landau equation is as follows\\
\begin{equation}\label{e.landau_collision}
\part_t f+v_i\part_{x_i}f=Q(f,f).
\end{equation}
$Q(f,f)$ is the collision kernel given by,\\
\begin{equation*}
Q(f,f)(v):=\part_{v_i}\int a_{ij}(v-v_*)(f(v_*)(\part_{v_j}f)(v)-f(v)(\part_{v_j}f)(v_*))\d v_*,
\end{equation*}
and $a_{ij}$ is the non-negative symmetric matrix defined by
\begin{equation}\label{e.matrix_a}
a_{ij}(z):=\left(\delta_{ij}-\frac{z_iz_j}{|z|^2}\right)|z|^{\gamma+2}.
\end{equation}

In all the expressions above (and in the rest of the paper), we use the convention that repeated lower case Latin indices are summed over $i,j=1,2,3.$

The quantity $\gamma$ encodes the type of interaction potential between the particles. Physically, the relevant regime is $[-3,1]$. We will be concerned with the regime $[0,1]$, that is the Maxwellian ($\gamma=0$) and all of the hard potentials.

For us it will be convenient to work with a slightly modified but equivalent version of \eref{landau_collision} which is as follows\\
\begin{equation}\label{e.landau}
\part_t f+v_i\part_{x_i}f=\a\part^2_{v_iv_j}f-\bar{c}f,
\end{equation}
where $c:=\part^2_{z_iz_j}a_{ij}(z)$, $\a:=a_{ij}*f$ and $\bar{c}:=c*f$.

In this paper, we give local in time solution to the Cauchy problem for the Landau equation, i.e. we study the solutions arising from prescribed regular initial data:
$$f(0,x,v)=f_{\ini}(x,v)\geq 0.$$

Our main result is that for sufficiently regular initial data, we have local existence and uniqueness for \eref{landau}. That is we have a unique, non-negative solution to the Cauchy problem for a small time.
\begin{theorem}\label{t.local}
Let $M_0>0$, $\gamma \in [0,1]$, $d_0>0$, $f_{\ini}$ be such that\\
\begin{equation*}
\sum\limits_{|\alpha|+|\beta|\leq 10}\norm{\jap{v}^{20-(\frac{3}{2}+\delta_\gamma)|\alpha|-(\frac{1}{2}+\delta_\gamma)|\beta|}\der(e^{d_0\jap{v}}f_{\ini})}_{L^2_xL^2_v}^2\leq M_0,
\end{equation*}
where,
\[   
\delta_\gamma= 
     \begin{cases}
       0 \hspace{1em} \text{if } \gamma\in[0,1)\\
	  \eta \text{ for some }\eta>0 \hspace{1em} \text{if } \gamma=1.
     \end{cases}
\]

Then for some $T>0$, depending on $\gamma, d_0$ and $M_0$, there is a non-negative solution, $f$ to \eref{landau} with $f(0,x,v)=f_{\ini}(x,v)$. 

In addition, $fe^{(d_0-\kappa t)\jap{v}}$, is unique in the energy space $\tilde E_T^4\cap C^0([0,T);\tilde Y^4_{x,v})$ ($\kappa$ is a large constant chosen later, that depends on $\gamma,d_0$ and $M_0$).\\
Moreover, $fe^{(d_0-\kappa t)\jap{v}}\in E_T\cap C^0([0,T);Y_{x,v})$ .
\end{theorem}

See \sref{notation} for the definition of spaces $E_T$, $Y_{x,v}$, $\tilde E^4_T$ and $\tilde Y^4_{x,v}$.

In the case of inhomogeneous equations, this is the first existence result for a binary collisional model featuring a hard and long-range potential. The major analytic difficulty in treating hard potentials for inhomogenous case stems from the fact that the velocity growth of the coefficients is too high to be taken care of by techniques employed for the soft potentials in \cite{AMUXY10,AMUXY13,HeSnTa17}. We overcome this by employing weighted (in velocity) energy norms that `penalize' the spatial and velocity derivatives differently (see \sref{diff} for more details).
\subsection{Related works} We now give a rather inexhaustive list of pertinent results that deal with a long range interaction potential.

\textbf{Local existence results.} One of the first relevant local existence results is due to the Alexandre--Morimoto--Ukai--Xu--Yang (AMUXY). The group proved in \cite{AMUXY10,AMUXY13} local existence for the non cut-off Boltzmann equation when the angular singularity parameter, $s$ is in the range $\left(0,\frac{1}{2}\right)$. More specifically, they prove for sufficiently regular initial data that is bounded by a Gaussian, there exists a local solution that stays bounded by a time dependent Gaussian. Moreover, they also prove $C^\infty$ smoothing for solutions that satisfy a non-vacuum condition.

For the Landau equation, Henderson--Snelson--Tarfulea proved the local existence, mass-spreading and $C^\infty$ smoothing of solutions to Landau equation with soft-potentials in \cite{HeSnTa17}. To prove local existence, they adapt the time-dependent Gaussian idea that appeared in \cite{AMUXY10, AMUXY13}.They also prove a mass-spreading theorem that helps them dispense with the non-vacuum contition required to prove $C^\infty$ smoothing.

Very recently, Henderson--Snelson--Tarfulea proved a local existence result for Boltzmann equation with $s\in (0,1)$ and $-\frac{3}{2}-2s\leq\gamma<0$ in \cite{HeSnTa19}. Specifically, they prove that if initial data is polynomially decaying in velocity, i.e. is bounded in a high weighted $L^\infty$ norm  and has five derivatives bounded in a weighted energy norm then there exists a unique and non-negative local in time solution.

\textbf{Global stability results for Landau equation.} The global regularity for the inhomogenous Landau is a well-known open problem. Most available global results are in the perturbative regime. Specifically, the stability of Maxwellians is quite well understood.

The stability problem of Maxwellians on a periodic box was established in Guo's seminal work, \cite {Guo02.1}. Since then Guo's nonlinear method has seen considerable success in understanding near Maxwellian regime \cite{Guo02,Guo03,Guo03.2,StGu04,StGu06,Guo12,StZh13}. Remarkably,this regime is well understood in the case of non-cutoff Boltzmann equation as well see, \cite{GrSt11,AMUXY12,AMUXY12.2,AMUXY12.3}.

Another perturbative result is the stability for vacuum. In contrast to Maxwelians, the stability problem for vacuum was only recently studied by Luk, in \cite{Lu18}, who proved the stability of vacuum in the case of moderately soft potentials ($\gamma \in (-2,0)$). Luk combined $L^\infty$ and $L^2$ methods to prove global existence of solutions near vacuum. Moreover, Luk shows that the main mechanism is dispersion by proving that the long-time limit of the solution to Landau equation solves the transport equation and that the solution does not necessarily approach global Maxwellians. An interesting feature of Luk's paper is the use of a hierarchy in his norm which is necessary for gaining enough time decay. We use a similar weighted norm inspired from his techniques.

\textbf{Hard Potentials for Landau.} The spatially homogeneous Landau equation with hard potentials was studied in detail by Desvillettes--Villani in \cite{DeVi00,DeVi00.2}, who showed existence and smoothness of solution with suitable initial data, as well as the appearance and propagation of various moments and lower bounds. The existence of a lower bound on the coefficient matrices independent of space was essential in proving such results and this is one reason why the existence theory for hard potentials in inhomogeneous Landau was missing from the literature.

For the inhomogeneous case Snelson proved in \cite{Sn18} that the solution to \eref{landau} with hard potentials (under assumptions of upper and lower bounded mass, energy and entropy) satisfies gaussian upper and lower bounds. The main feature of the paper is the appearance of these bounds which is reminiscent of the spatially homogenous case.
\subsection{Future Directions} 
We now outline some open problems that could be interesting for a future work.
\begin{enumerate}
\item \textbf{Propogating Gaussian bounds instead of exponential ones.} We propogate exponential bounds instead of gaussian decay, which is in contrast to the local theory for soft potentials (see \cite{HeSnTa17, AMUXY10, AMUXY13}). It would be interesting to see if one can indeed propogate the expected Gaussian decay. For a more detailed discussion on this issue see \sref{diff}.

\item \textbf{$C^\infty$ smoothing of solutions to hard potentials.} A regular enough solution to \eref{landau}with hard potentials is expected to become instantaeously smooth as in the the case of soft potentials (see \cite{HeSnTa17}). It also seems plausible that a mass-spreading theorem (as in \cite{HeSnTa17}) also remains true for the hard potentials. In that case one could expect local in space gaussian upper and lower bounds because of the recent work \cite{Sn18} by Snelson.

\item \textbf{Local existence for non cut-off Boltzmann equation.} Until recently, the state of the art results for the local existence of solutions to non cut-off Boltzmann, \cite{AMUXY10,AMUXY13}, was due to the AMUXY group. In these results they can handle angular singularity, s, in the range $\left(0,\frac{1}{2}\right)$ and $\gamma+2s<1$.

Henderson--Snelson--Tarfulea, proved a local existence result for Boltzmann in \cite{HeSnTa19} for $s\in(0,1)$ and $-\frac{3}{2}-2s\leq\gamma<0$. This result is comparable to the soft potential result of the same authors in \cite{HeSnTa17}. A striking feature of the paper is that they can propagate polynomial bounds in contrast to Gaussian bounds for the case of Landau. The main reason why this is possible is that we have fractional derivatives (corresponding to $s$) for Boltzmann while we have full derivatives for Landau.

The authors of \cite{HeSnTa19} nicely leverage this fact and although they necessarily have to bound an $L^2_v$ term with higher velocity weights they can handle, the number of derivatives is a little lower than that for Landau. Thus this gain in derivatives makes up for the moment loss via an interpolation between polynomially-weighted $L^2$ estimates and polynomially-weighted $L^\infty$ estimates.

Landau equation can be considered ,at least formally, a limit of Boltzmann equation in the grazing collision limit, i.e. the limit as $s\to 1$. This fact in addition to the comparison between \cite{HeSnTa17} and \cite{HeSnTa19} strongly suggests that the velocity weight (or moment loss) issue is most severe for the Landau equation. So we are hopeful that the ideas developed in this paper will be useful in resolving the local existence issues in case of Boltzmann, especially for the case $0\leq \gamma\leq1$. We would also like to remark that the ideas developed here especially a hierarchical $L^2$ based approach might be useful in dispensing with the weighted $L^\infty$ norm that is required in \cite{HeSnTa19}. 

We finally note that the local existence question for Boltzmann equation with $\gamma\in (-3,\max(-3/2-2s,-3)]$ and $\gamma>0$ is still open.

\item \textbf{Global existence and stability of vacuum.} The stability of vacuum problem entails proving global existence for small data and that these solutions converge to a solution of a linear transport equation in the limit $t\to \infty$ (see \cite{Lu18} for more details). The stability problem for hard potentials already seems tractable with a slight variation in the ideas presented in this paper and will be treated in a future work.

We would also like to point out that the stability of vacuum problem is open for the non cut-off Boltzmann equation for any physically relevant regime.
\end{enumerate}
\subsection{Paper organization} The remainder of the paper is structured as follows.

In \sref{diff} we give an overview of the additional difficulties one has to overcome in the case of hard potentials and a proof strategy. In \sref{notation}, we introduce some notations that will be in effect throughout the paper. In \sref{set_up} we set-up the energy estimate and the error terms that we need to estimate. In \sref{coefficient_bounds}, we prove estimates for the coefficient matrix $\a$, $\cm$ and its higher derivatives. In \sref{errors}, we bound the error terms and finally in \sref{put}, we put everything together to prove \tref{local}.
\subsection{Acknowledgements} I would like to thank Jonathan Luk for suggesting this problem, for countless helpful discussions and for constant encouragement. I would also like to thank Panagiotis Dimakis for various stimulating discussions and for his feedback on the manuscript. In addition, I am grateful to Scott Armstrong and Cole Graham for their comments on an earlier version of the paper.
\section{Difficulties in hard potential and Proof strategy}\label{s.diff}
In this section we give a brief presentation of the ideas involved in resolution for existence for soft potentials, new difficulties that arise in the case of hard potentials, why they cannot be resolved with the techniques employed for the soft potentials in \cite{HeSnTa17} and the resolutions we propose.
\subsection{Review of the soft potential case} The existence theory of soft potentials for Landau equation was recently proved in \cite{HeSnTa17} by Henderson--Snelson--Tarfulea. Even at the level of soft potentials, the velocity growth issue for the coefficients is non-trivial to deal with. To see this, we fix a $\gamma\in (-2,0)$ and differentiate \eref{landau} by $\part^\alpha_x$, to get an equation of the form,
\begin{equation}\label{e.lan_ill}
\part_t\part^\alpha_x f+v_i\part_{x_i}f=\part^{\alpha}_x(\a \part^2_{v_iv_j}f)-\part^\alpha_x\cm f.
\end{equation}

To illustrate the issue, we focus on the term $\part^\alpha \a (\part^2_{v_iv_j}f)$. More specifically, for any $\gamma\in (-2,0)$, we have the bound $a_{ij}(v-v_*)\lesssim \jap{v}^{2+\gamma}\jap{v_*}^{2+\gamma}$ (see \pref{pointwise_estimates_a}) thus the best pointwise bound we can hope for is, $$|\part^\alpha_x\a|(t,x,v)\lesssim |v|^{2+\gamma}.$$

With this bound we cannot hope to close the estimates, since there is no term on the left hand side of \eref {lan_ill} that can absorb a term with so strong a growth in velocity. To overcome this, Henderson--Snelson--Tarfulea restrict the initial data class to functions bounded by a Gaussian and propogate bounds with the aid of a time dependent Gaussian that flattens out over time. This technique is inspired by earlier works of the AMUXY group on local existence of non-cutoff Boltzmann in \cite{AMUXY10,AMUXY13}.

Thus they consider the equation that $g=fe^{-(d_0-\kappa t)\jap{v}^2}$ satisfies, by plugging this ansatz into \eref{landau} to get,
\begin{equation}\label{e.fail_1}
\begin{split}
\partial_tg+v_i\partial_{x_i}g+\kappa\jap{v}^2g&=\a [f]\partial^2_{v_iv_j}g-\cm [f]g-2d(t)\a [f]\jap{v}\partial_{v_j} g
\\& \quad-d(t)\left(\delta_{ij}-2d(t)v_iv_j\right)\a [f]g.
\end{split}
\end{equation}

The term $\kappa\jap{v}^2g$ arises due to differentiation of the time dependent Gaussian and is key to helping us take care of the problem discussed before and close the estimates. Although, it is not clear a priori, but due to the anisotropy of the matrix $a$, they are able to control the term involving $\a v_iv_j$. For more details on this see \cite{HeSnTa17} or \pref{pointwise_estimates_a}.
\subsection{The new difficulties of hard potentials}\label{ss.sad} The main reason why the aforementioned techniques are not very useful for handling the hard potentials case is that the coefficient matrix has too strong a growth to even be tamed by the time dependent Gaussian. We specialize to the case when $\gamma=1$ but the same issues are present for the whole hard potentials regime.\\
In this case, $$a_{ij}(z)=\left(\delta_{ij}-\frac{z_iz_j}{|z|^2}\right)|z|^{3}.$$ 

Again consider \eref{fail_1}. Applying $\part^\alpha_x\part^\beta_v$, with $|\alpha|=|\beta|=2$, to \eref{fail_1} we get,
\begin{equation}\label{e.fail_2}
\part_t \der g+v_i\part_{x_i}\der g+\kappa\jap{v}^2\der g=\der (\a\part^{v_iv_j}g)-\der (\cm g)+\text{ other terms}.
\end{equation}

To set up energy estimates, we multiply \eref{fail_2} by $\der g$ and integrate in time, space and velocity. To understand why this falls short, we look at the term, $$I=\left|\int_0^T\int_x\int_v \part^\alpha_x\part^\beta_v g(\part^\alpha_x \a)\part^\beta_v\part^2_{v_iv_j}g\d v\d x\d t\right|.$$

For the case of $\gamma=1$, we again have the bound $a_{ij}(v-v_*)\lesssim \jap{v}^3\jap{v_*}^3$ (see \pref{pointwise_estimates_a}) thus, $$|\part^\alpha_x\a|(t,x)\lesssim \jap{v}^3.$$

Since both terms involving derivatives of $g$ has four derivatives, we need to estimate them in $L^2_xL^2_v$. We thus have the following estimate,
$$I\lesssim \norm{\jap{v}^{\frac{3}{2}}\der  g}_{L^2([0,T];L^2_xL^2_v)}\norm{\jap{v}^{\frac{3}{2}}\part^\beta_v\part^2_{v_iv_j}  g}_{L^2([0,T];L^2_xL^2_v)}.$$

Now the best term on the left hand side is of the form $\norm{\jap{v}\part^\alpha_x\part^\beta_v  g}_{L^2([0,T];L^2_xL^2_v)}$, we have no hope of closing this estimate.
\subsection{A hopeful term}\label{ss.hope} Although things look quite bleak, a trivial observation is quite instrumental in coming up with the hierarchy of weighted norms. That is, whenever we have a spatial derivative hitting $\a$, then we have one less spatial derivative hitting on $\part^2_{v_iv_j}g$. Thus we would like to devise a weight function that uses this information.

Moreover we also note that taking velocity derivatives of $\a$ is helpful. Specifically, we again consider the example from last subsection but this time, assume both velocity derivatives fall on $\a$ (and say none of the spatial ones do). This term also appears due to integration by parts at the highest order.\\
Let, $$\tilde I=\left|\int_0^T\int_x\int_v \part^\alpha_x\part^\beta_v g(\part^\beta_v \a)\part^\alpha_x\part^2_{v_iv_j}g\d v\d x\d t\right|.$$

For the case of $\gamma=1$, we have the bound $\part^\beta_{v}a_{ij}(v-v_*)\lesssim \jap{v}\jap{v_*}$ (see \pref{pointwise_estimates_a}) thus, $$|\part^\beta_v\a|(t,x,v)\lesssim \jap{v}.$$

Thus we get the following bound,
\begin{align*}
\tilde I&\lesssim \norm{\jap{v}^{\frac{1}{2}}\part^\alpha_x\part^\beta_v  g}_{L^2([0,T];L^2_xL^2_v)}\norm{\jap{v}^{\frac{1}{2}}\part^\beta_v\part^2_{v_iv_j}  g}_{L^2([0,T];L^2_xL^2_v)}.
\end{align*}

This time around, we can absorb, this term on the left hand side by the term of the form  $\norm{\jap{v}\der  g}_{L^2([0,T];L^2_xL^2_v)}$. In fact, we can get away with just $\jap{v}^{\frac{1}{2}}$ weight in the norm (this will be useful later).
\subsection{A hierarchy of weighted norms}\label{ss.hie} As a lesson from above examples, we take away that taking velocity derivatives of the coefficient matrix is `good' and taking spatial derivatives does not do anything.

So a way to get over this problem might be to consider weighted Sobolev norms where we `penalize' taking spatial derivatives in some way, i.e we should expect to bound the terms with more spatial derivatives in a weaker weighted norm.

To motivate our proposed hierarchy of weighted norms, we work formally with \eref{fail_2}.

Again, we multiply by $\der g$ and integrate in $t,x$ and $v$ and we specialize to the case $\gamma=1$.\\
The need for a hierarchy of velocity weights can be seen with the aid of $\der(\a \part^2_{v_iv_j}g)$ term.

There are a few cases that will show us why we need the weights,
\begin{enumerate}
\item When no derivatives fall on $\a$ then in the energy estimate we get a term of the form $$\int_0^T \int \int \der g(\a)\part^2_{v_iv_j} \der g \d v \d x \d t.$$
In this case we necessarily have to perform integration by parts and we get two terms $-\int_0^T \int \int \part_{v_i}\der g(\a)\part_{v_j} \der g \d v \d x \d t$ which we can ignore assuming $f$ is non-negative and the other term is $\int_0^T \int \int \der g(\part^2_{v_iv_j}\a) \der g \d v \d x \d t$.\\
We have already encountered this term in the \ssref{hope} and as before, this term poses no problem.
\item Now assume we have two spatial derivatives hitting $\a$. That is we have a term of the form
$$\int_0^T \int \int \der g(\part^2_{x}\a)\part^2_{v_iv_j}\derv{''}{''} g \d v \d x \d t.$$
We have also seen this term before in \ssref{sad} as $I$ and it is one of the terms that prevents us from closing the estimates with the soft potential methodology.\\
Thus inspired from the hierarchy of $\jap{x-tv}$ weights in \cite{Lu18} we use velocity weights dependent on number of derivatives hitting $f$. If we had a weight function $\bar{\omega}_{\alpha,\beta}=20-2|\alpha|$, then performing the energy estimates with $\jap{v}^{2\bar{\omega}_{\alpha,\beta}}\der g$ we have a hope of closing the estimates.

Indeed the term $\part^2_{v_iv_j}\derv{''}{''} g$ has two less spatial derivatives than $\der g$ and hence can handle two extra velocity weights which knocks down the unaccounted for velocity weights to $\jap{v}$ as in the case 1.
\item Unfortunately, this hierarchy fails for a different configuration of derivatives. Assume that we hit the equation with $10$ velocity derivatives and look at the term when all the derivatives hit $\a$. We abuse the multi-index notation by writing $\part^{10}_v=\part^\beta_v$ for some $\beta$ such that $|\beta|=10$.\\
The term reads
$$\int_0^T \int \int  \jap{v}^{40}\part^{10}_v g(\part^{10}_v\a)\part^2_{v_iv_j} g \d v \d x \d t.$$

Since $\a$ and the first $g$ is being hit by top order derivatives, we have no choice but to estimate them in $L^2_x$, leaving us to estimate $\part^2_{v_iv_j} g$ in $L^\infty_x$. Since the number of derivatives hitting the second $f$ is only two we can use Sobolev embedding to estimate it in $L^2_x$. But then we are generating spatial derivatives and thus this term can not take as many as $20$ velocity weights.

This means that we need to penalize the number of velocity derivatives in our hierarchy as well. To deal with a problem we created a different problem!

Hence we need to come up with a hierarchy that penalizes spatial derivatives as well as velocity derivatives. But this penalty should be skewed in the direction of spatial derivatives because of the term of the following type $$\int_0^T \int \int  \jap{v}^{2\bar{\omega}_{\alpha,\beta}}\part^{10}_x g(\part^{2}_x\a)\part^2_{v_iv_j} \part^8_x g \d v \d x \d t.$$
Because we need the term $\part^2_{v_iv_j}\part^8_x f $ to handle two more $\jap{v}$ weights coming from $\part^2_x \a$ as before. Here we again abused the multi-index notation.
\end{enumerate}
With these things in mind, we claim that the hierarchy $\omega_{\alpha,\beta}=20-\frac{3}{2}|\alpha|-\frac{1}{2}|\beta|$ works. Indeed, in our example when only two spatial derivatives hit $\a$ then the weight $\part^2_{v_iv_j}\derv{''}{''} g$ can handle $20-\frac{3}{2}(|\alpha|-2)-\frac{1}{2}2$ weights, which is exactly two more than $\omega_{\alpha,\beta}$. 

Next, the case when all the derivatives are velocity derivatives, we again look at the term
$$\int_0^T \int \int  \jap{v}^{2\omega_{\alpha,\beta}}\part^{10}_v g(\part^{10}_v\a)\part^2_{v_iv_j} g \d v \d x \d t.$$ 
In this case $\omega_{\alpha,\beta}=15$ and $\part^2_x\part^2_{v_iv_j} g$ (which we get after applying Sobolev embedding in $x$) can handle $16$ $\jap{v}$ weights.\\
We claim that this hierarchy works and is used to estimate the error terms in \sref{errors}.
\begin{remark}
Note that with this definition of weight function, we need to work with more derivatives than is known for soft-potentials (see \cite{HeSnTa17,AMUXY13}). The reason is that when $\a$ is hit with top or next to top order derivatives then we need to estimate it in $L^2_x$ and the accompanying term in $L^\infty_x$ which is ultimately bounded in $L^2_x$ at the cost of two spatial derivatives and more crucially $\jap{v}^3$ weights.
\end{remark}
\subsection{Why we use exponential instead of Gaussian?}
Now, we discuss why we propagate exponential bound and not Gaussian. One reason is that since $\gamma\in [0,1]$, a time dependent exponential is enough. Second reason, is a bit more subtle, this has mostly to do with the term of the form $\a v_iv_j$ (this is what one would get instead of $\frac{v_iv_j}{\jap{v}^2}\a$ in \eref{eq_for_g} if one tried to propagate a gaussian bound, see \eref{lan_ill}).

Again we specialize to $\gamma=1$. We remind ourselves of the pointwise bound from \cite{Lu18} (Prop. 5.7), $$|\der(\a v_iv_j)|(t,x,v)\lesssim \int (|v|^{3}|v_*|^2+|v_*|^5)|\der f|(t,x,v_*)\d v_*.$$

This can also be obtained by adapting our proof of \eref{pw_bound_a_2v}. Now we see that we are in the same fix as in \ssref{sad}. There is no way of absorbing this on the LHS and hence no hope to close the estimates. On the other hand, if we try to propagate the exponential bound this issue goes away and we are able to almost close the estimates.

Propagating an exponential in conjunction with the weighted hierarchy of norms takes care of almost all the terms. There is still one problematic term that posed no issues in case of soft potential. Namely, $$-d(t)\int_0^T \int \int \jap{v}^{2\omega_{\alpha,\beta}}(\der g)^2 \frac{\bar{a}_{ii}}{\jap{v}}\d v\d x\d t.$$
This term arises when both $\part^2_{v_iv_j}$ hit $e^{-d(t)\jap{v}}$ and none hits $g$ (see \eref{eq_for_g}). Since $|\bar{a}_{ii}|\lesssim \jap{v}^3$ for $\gamma=1$ (see \pref{pointwise_estimates_a}), we have no hope of closing this estimate as the `good' term that arises by propagating an exponential can only handle an additional $\jap{v}$ weight.
But assuming $f$ is non-negative, we note that $\bar{a}_{ii}$ is non-negative and hence the whole integral has a good sign. Thus we can safely ignore this term. See \lref{J_term} for more details.
\subsection{Why we need $\delta_\gamma>0$ for $\gamma=1$?} We finally comment on why we need to choose a $\delta_\gamma>0$ for the case of $\gamma=1$ (see \tref{local}).

To prove existence of local solutions, we follow the strategy employed in \cite{HeSnTa17}. More concretely, we first show existence to a linearized problem and show that there is a solution to the full nonlinear problem using an iteration argument.

To show existence for a linearized equation we use the vanishing viscosity method. That is we add, a small viscosity to the linearized equation and solve the Cauchy problem on a bounded but arbitrary size domain. The idea is to get estimates independent of viscosity and the domain and take a weak limit in an appropriate norm and prove that this weak limit satisfies the linearized equation.

Since we solve our approximate linear equation on a bounded domain, there are some boundary error terms that we need to make sure vanish as we increase the size of the domain to infinity.

To make sure that these terms do decay, we need to modify the hierarchy in case of $\gamma=1$.\\
More specifically, we solve our linearized equation on $B_R$ (in phase space) and add $\eps\Delta_{x,v} \der g$ to leverage the existence theory of parabolic equations.

Since we solve on a bounded domain, we use cut-off functions to stay away from the boundary. We introduce a hierarchy of cut-off functions $\psi_m$ (where $m=|\alpha|+|\beta|$) in the spirit of \cite{HeSnTa17}.

So for energy estimates, we multiply by $\psi_m^2\jap{v}^{2\omega_{\alpha,\beta}}\der g$ instead of just $\jap{v}^{2\omega_{\alpha,\beta}}\der g$. The presence of these cut-off functions give rise to boundary terms when we perform integration by parts.

Note that the extra viscous term (after the required integration by parts) gives us a term on the left hand side of the form,
\begin{equation}\label{e.vis}
\norm{\jap{v}^{2\omega_{\alpha,\beta}}\psi_m\derv{'}{'} g}_{L^2([0,T];L^2_xL^2_v)}.
\end{equation}
Here $|\alpha'|+|\beta'|=|\alpha|+|\beta|+1$ and either $|\alpha'|=|\alpha|+1$ or $|\beta'|=|\beta|+1$. Now, by our definition of $\omega_{\alpha,\beta}$ in \ssref{hie} we see that $2\omega_{\alpha',\beta'}+1\leq 2\omega_{\alpha,\beta}$ (since we have at least one extra derivative). Since we bound this term in a higher weighted norm, we use this to show smallness of the boundary error terms via an induction.

But a typical boundary error term is of the form, 
\begin{equation}\label{e.bdy}
\left|\int_0^T\int \int \jap{v}^{2\omega_{\alpha,\beta}}\a\psi_m\part^2_{v_iv_j}\psi_m(\der g)^2 \d v\d x\d t\right|.
\end{equation}
We arrange our cut-off so that $|\part^2_{v_iv_j}\psi_m|\lesssim R^{-2}\psi_{m-1}$ and since we are in $B_R$, we have $\jap{v}\lesssim R$.\\
Using this, $|\a|\lesssim \jap{v}^{2+\gamma}$ (see \ssref{sad}) and Young's inequality we get,
\begin{align*}
\eref{bdy}\lesssim R^{\gamma-1}[\norm{\jap{v}^{\omega_{\alpha,\beta}}\jap{v}^{\frac{1}{2}}\psi_m\der g}_{L^2([0,T];L^2_xL^2_v)}+[\norm{\jap{v}^{\omega_{\alpha,\beta}}\jap{v}^{\frac{1}{2}}\psi_{m-1}\der g}_{L^2([0,T];L^2_xL^2_v)}].
\end{align*}

The term from the time dependent exponential takes care of the first term and we hope to use \eref{vis} to take of the second term. For $\gamma\in[0,1)$, we get smallness since we are taking a negative power of $R$, but for $\gamma=1$ we just get boundedness.

So if we change the hierarchy a little bit for $\gamma=1$, that is if we have $$\omega_{\alpha,\beta}=20-|\beta|(\frac{1}{2}+\eta)-|\alpha|(\frac{3}{2}+\eta),$$
for any $\eta>0$, we see that in \eref{vis}, $2\omega_{\alpha',\beta'}+1+2\eta\leq 2\omega_{\alpha,\beta}$.

Now we can do the boundary estimate a bit differently to get,
\begin{align*}
\eref{bdy}\lesssim R^{\gamma-1-\eta}[\norm{\jap{v}^{\omega_{\alpha,\beta}}\jap{v}^{\frac{1}{2}}\psi_m\der g}_{L^2([0,T];L^2_xL^2_v)}+[\norm{\jap{v}^{\omega_{\alpha,\beta}}\jap{v}^{\frac{1}{2}+\eta}\psi_{m-1}\der g}_{L^2([0,T];L^2_xL^2_v)}].
\end{align*}
Thanks to our new hierarchy of weights, we can both get smallness of the boundary terms and use the viscosity term to absorb the second term above on the left hand side (formally achieved by induction).
\section{Notations and spaces}\label{s.notation}
We introduce some notations that will be used throughout the paper.

\textbf{Norms}: We will use mixed $L^p$ norms, $1\leq p<\infty$ defined in the standard way:\\
$$\norm{h}_{L^p_v}:=(\int |h|^p \d v)^{\frac{1}{p}}.$$
For $p=\infty$, define
$$\norm{h}_{L^\infty_v}:=\text{ess} \sup_{v\in \R^3}|h|(v).$$

For mixed norms, the norm on the right is taken first. For example,
$$\norm{h}_{L^p_xL^q_v}:=(\int_{\R^3}(\int_{\R^3} |h|^q(x,v) \d v)^{\frac{p}{q}} \d x)^{\frac{1}{p}},$$
and
$$\norm{h}_{L^r([0,T];L^p_xL^q_v)}:=(\int_0^T(\int_{\R^3}(\int_{\R^3} |h|^q(x,v) \d v)^{\frac{p}{q}} \d x)^{\frac{r}{p}})^{\frac{1}{r}},$$
with obvious modifications when $p=\infty$, $q=\infty$ or $r=\infty$.\\

\textbf{Japanese brackets}. Define $$\jap{\cdot}:=\sqrt{1+|\cdot|^2}.$$

\textbf{Multi-indices}. Given a multi-index $\alpha=(\alpha_1,\alpha_2,\alpha_3)\in (\N\cup \{0\})^3$, we define $\part_x^\alpha=\part^{\alpha_1}_{x_1}\part^{\alpha_2}_{x_2}\part^{\alpha_3}_{x_3}$ and similarly for $\part^\beta_v$. Let $|\alpha|=\alpha_1+\alpha_2+\alpha_3$. Multi-indices are added according to the rule that if $\alpha'=(\alpha'_1,\alpha'_2,\alpha'_3)$ and $\alpha''=(\alpha''_1,\alpha''_2,\alpha''_3)$, then $\alpha'+\alpha''=(\alpha'_1+\alpha''_1,\alpha'_2+\alpha''_2,\alpha'_3+\alpha''_3)$.\\

\textbf{Velocity weights}. We define the velocity weight function that we use in our energy norm. Let $|\alpha|+|\beta|\leq 10$ we define $$\omega_{\alpha,\beta}:=20-(\frac{3}{2}+\delta_\gamma)|\alpha|-(\frac{1}{2}+\delta_\gamma)|\beta|,$$
where
\[   
\delta_\gamma= 
     \begin{cases}
       0 \hspace{1em} \text{if } \gamma\in[0,1)\\
	  \eta \text{ for some }\eta>0 \hspace{1em} \text{if } \gamma=1.
     \end{cases}
\]
\\
\textbf{Global energy norms}. We now describe the energy norm we use in $[0,T)\times\R^3\times\R^3$. 
\begin{equation*}
\norm{h}^2_{E_T^m}:=\sum \limits_{|\alpha|+|\beta|\leq m} \norm{\jap{v}^{\omega_{\alpha,\beta}}\der h}^2_{L^\infty([0,T);L^2_xL^2_v)}+\sum \limits_{|\alpha|+|\beta|\leq m} \norm{\jap{v}^{\frac{1}{2}}\jap{v}^{\omega_{\alpha,\beta}}\der h}^2_{L^2([0,T);L^2_xL^2_v)},
\end{equation*}
when $m=10$, it is dropped from the superscript.

It will also be convenient to define some other energy type norms. Namely,
$$\norm{h}_{Y^m_v}^2(t,x):=\sum \limits_{|\alpha|+|\beta|\leq m} \norm{\jap{v}^{\omega_{\alpha,\beta}}\der h}^2_{L^2_v}(t,x),$$
$$\norm{h}^2_{Y^m_{x,v}}(t):=\sum \limits_{|\alpha|+|\beta|\leq m} \norm{\jap{v}^{\omega_{\alpha,\beta}}\der h}^2_{L^2_xL^2_v}(t),$$
$$\norm{h}^2_{Y^m_{T}}:=\sum \limits_{|\alpha|+|\beta|\leq m} \norm{\jap{v}^{\omega_{\alpha,\beta}}\der h}^2_{L^\infty([0,T];L^2_xL^2_v)}.$$
and 
$$\norm{h}^2_{X^m_{T}}:=\sum \limits_{|\alpha|+|\beta|\leq m} \norm{\jap{v}^{\frac{1}{2}}\jap{v}^{\omega_{\alpha,\beta}}\der h}^2_{L^2([0,T];L^2_xL^2_v)}.$$

Finally we define the the weaker energy norms used in statement of \tref{local} for uniqueness of solutions.
\begin{equation*}
\norm{h}^2_{\tilde E_T^4}:=\sum \limits_{|\alpha|+|\beta|\leq 4} \norm{\jap{v}^{\tilde{\omega}_{\alpha,\beta}}\der h}^2_{L^\infty([0,T);L^2_xL^2_v)}+\sum \limits_{|\alpha|+|\beta|\leq 4} \norm{\jap{v}^{\frac{1}{2}}\jap{v}^{\tilde{\omega}_{\alpha,\beta}}\der h}^2_{L^2([0,T);L^2_xL^2_v)},
\end{equation*}
and $$\norm{h}^2_{\tilde Y^4_{x,v}}(t):=\sum \limits_{|\alpha|+|\beta|\leq 4} \norm{\jap{v}^{\tilde \omega_{\alpha,\beta}}\der h}^2_{L^2_xL^2_v}(t),$$
where we define $\tilde \omega_{\alpha,\beta}=10-\left(\frac{3}{2}+\delta_\gamma\right)|\alpha|-\left(\frac{1}{2}+\delta_\gamma\right)|\beta|$, for $|\alpha|+|\beta|\leq 4$.\\

\textbf{Local energy norms}. Since we need to carry out the viscosity method estimates in a bounded domain, we need local versions of the energy norms. We defer the definitions till \sref{errors}.\\

For two quantitites, $A$ and $B$ by $A\lesssim B$, we mean $A\leq C(d_0,\gamma)B$, where $C(d_0,\gamma)$ is a positive constant depending only on $d_0$ and $\gamma$.
\section{Set-up for the energy estimates}\label{s.set_up}
Since we plan to propagate a time dependent exponential bound we start by defining $g:=e^{d(t)\jap{v}}f$.\\
 Here $d(t)=d_0-\kappa t$, $t\in (0,T_0]$ and $T_0=\frac{d_0}{2\kappa}$($\kappa$ is some large constant, fixed later).\\
Substituting $f=e^{-d(t)\jap{v}}g$ in \eref{landau}, we obtain the equation that $g$ satisfies\\
\begin{equation} \label{e.eq_for_g}
\begin{split}
\partial_tg+v_i\partial_{x_i}g+\kappa\jap{v}g&=\a [f]\partial^2_{v_iv_j}g-\cm [f]g-2d(t)\a [f]\frac{v_i}{\jap{v}}\partial_{v_j} g
\\&\quad -d(t)\left(\frac{\delta_{ij}}{\jap{v}}-\left(d(t)+\frac{1}{\jap{v}}\right)\frac{v_iv_j}{\jap{v}^2}\right)\a [f]g.
\end{split}
\end{equation}

As outlined in introduction, we use an iteration argument to prove existence of solutions to \eref{eq_for_g}. Thus as a first step we work towards proving a linearized version of \eref{eq_for_g}
\begin{lemma}\label{l.lin_eq}
Let $M_h>0$, $T\in (0,T_0]$, $g_{\ini}$ and $h$ be given nonnegative functions. Also let $g_{\ini}$ be such that $\norm{g_{\ini}}_{Y_{x,v}}<M_0$ and $h$ be such that $\norm{he^{d(t)\jap{v}}}_{Y_T}<M_h$. Then there exists a solution to the linearized problem 
\begin{equation} \label{e.lin_eq_landau}
\begin{split}
\partial_t\bar G+v_i\partial_{x_i}\bar G+\kappa\jap{v}\bar G&=\a [h]\partial^2_{v_iv_j}\bar G-\cm [h]\bar G-2d(t)\a [h]\frac{v_i}{\jap{v}}\partial_{v_j} \bar G
\\&\quad-d(t)\left(\frac{\delta_{ij}}{\jap{v}}-\left(d(t)+\frac{1}{\jap{v}}\right)\frac{v_iv_j}{\jap{v}^2}\right)\a [h]\bar G,
\end{split}
\end{equation}
with $\bar G(0,x,v)=g_{\ini}(x,v)$ and $\kappa$, a large enough constant. Moreover, $\bar G$ is non-negative and 
\begin{equation}\label{e.lin_exp_bound}
\norm{\bar G}^2_{E_T}\leq \norm{g_{\ini}}_{Y_{x,v}}^2\exp{(C(d_0,\gamma,\kappa,M_h)T)}.
\end{equation}
\end{lemma}
The proof will be concluded in \sref{put}.

As an intermediate step to proving \lref{lin_eq} we use vanishing viscosity method. We begin by noting existence for equation that is obtained by adding a small viscosity in \eref{eq_for_g}. To appeal to the standard parabolic theory we smoothen out the initial data near boundary of our domain and make sure that our linearization is also smooth.
In the following lemma and henceforth we use the following notation: for any $\eps>0$ and $R>3$, we define the mollifier $\zeta_\eps(x,v)=\eps^{-6}\zeta(x/\eps,y/\eps)$ for some non-negative, $C^\infty_c$ function $\zeta$ such that $\int \zeta \d v \d x=1$. We work on a ball $\Omega_R:=\{(x,v)\in \R^6:|x|^2+|v|^2<R^2\}$. Moreover, let $\chi_L$ be a smooth cut-off function on $\R^6$, supported on $\Omega_{L-1}$, equal to $1$ in $\Omega_{L-2}$, radially symmetric, monotone and such that for $|\alpha|+|\beta|=n$, $|\der \chi_R|<2^n$. 
\begin{lemma}\label{l.lin_visc_landau_exis}
Let $g_{\ini}$ and $h$ be given nonnegative functions with $T>0$. Also let $g_{\ini}$ be such that $\norm{g_{\ini}}_{Y_{x,v}}<\infty$ and $h$ be such that $\norm{he^{d(t)\jap{v}}}_{\H}<\infty$. For any $\eps>0$, let $h_\eps=\zeta_\eps*h$. Then for $R$ sufficiently large, there exists a unqiue solution $G=G_{h,R,\eps}$ to 
\begin{multline}\label{e.lin_eq_visc_landau}
\partial_tG+v_i\partial_{x_i}G+\kappa\jap{v}G=\eps\part^2_{x_i,x_i}G+\eps\part^2_{v_i,v_i}G+\a [h_\eps]\partial^2_{v_iv_j}G-\cm [h_\eps]G-2d(t)(\a [h_\eps]+\eps\delta_{ij})\frac{v_i}{\jap{v}}\partial_{v_j} G
\\-d(t)\left(\frac{\delta_{ij}}{\jap{v}}-\left(d(t)+\frac{1}{\jap{v}}\right)\frac{v_iv_j}{\jap{v}^2}\right)(\a [h_\eps]+\eps\delta_{ij})G
\end{multline}
on $[0,T]\times G$ such that \\
\begin{equation}
\begin{split}
&G(0,x,v)=\chi_R(x,v)(\zeta_\eps*g_{\ini})(x,v)\\
& G(t,x,v)=0 \text{ for all }(t,x,v)\in [0,\infty)\times \partial \Omega_R
\end{split}
\end{equation}

Moreover $G$ is nonnegative and $G\in C^\infty([0,T]\times \Omega_R).$
\end{lemma}
\begin{proof}
Existence of such $G$ follows from standard parabolic theory. Non-negativity is a bit tricky, because in this form it's not clear if we can apply maximum principle.

To get over that issue, we conider the equation that $F=G e^{-d(t)\jap{v}}$ satisfies. We just get \eref{landau} back with an additional viscosity term. More precisely, we have
\begin{equation}\label{e.landau_with_visc}
\part_{t}F+v_i\part_{x_i} F=\eps \Delta_{x,v} F+\a[h_\eps]\part^2_{v_iv_j}F-\cm[h_\eps] F.
\end{equation}
Thanks to positivity of $h_\eps$, we have that $-\cm[h_\eps]$ is positive and thus by maximum principle we get that $F$ is non-negative as the boundary conditions are positive. Since $G=Fe^{d(t)\jap{v}}$, it is non-negative as well.
\end{proof}
Since we have $\norm{\der h_\eps}_{L^p}\leq \norm{\der h}_{L^p}$, we suppress its dependence on $\eps$ and just use the bounds for $h$.

Applying $\part_x^\alpha \part^\beta_v$ to \eref{lin_eq_visc_landau}, we get
\begin{equation} \label{e.lin_diff_visc_eq}
\begin{aligned}
&\partial_t \partial^\alpha_x\part_v^\beta G+v_i\partial_{x_i}\partial^\alpha_x\part_v^\beta G+\kappa\jap{v}\partial^\alpha_x\part_v^\beta G-\a[h]\part^2_{v_iv_j}\der G-\eps \part^2_{x_i,x_i}\der G-\eps \part^2_{v_i,v_i}\der G \\
&=\underbrace{[\partial_t+v_i\partial_{x_i},\partial_x^\alpha\partial_v^\beta ]G}_{\text{Term } 1}+\underbrace{\kappa(\part^\alpha_x \part_v^\beta (\jap{v}G)-\jap{v}\part^\alpha_x \part_v^\beta  G)}_{\text{Term } 2}\\
&\quad+\underbrace{\part^\alpha_x \part_v^\beta  (\a[h] \part^2_{v_iv_j}G)-\a[h]\part^2_{v_iv_j}\der G}_{\text{Term } 3}-\underbrace{\part^\alpha_x \part_v^\beta(\cm[h] G)}_{\text{Term } 4}-\underbrace{2(d(t))\part^\alpha_x \part_v^\beta \left((\a[h]+\eps\delta_{ij}) \frac{v_i}{\jap{v}}\part_{v_j}G\right)}_{\text{Term } 5}\\
&\quad-\underbrace{d(t)\part^\alpha_x \part_v^\beta \left(\left(\frac{\delta_{ij}}{\jap{v}}-\left(d(t)+\frac{1}{\jap{v}}\right)\frac{v_iv_j}{\jap{v}^2}\right)(\a[h]+\eps\delta_{ij}) G\right)}_{\text{Term } 6}.
\end{aligned}
\end{equation}

Now \lref{lin_visc_landau_exis} only provides zero boundary conditions for $G$ and not for higher derivatives. Thus, we need to multiply by a cut-off function in velocity that vanishes on $\part \Omega_R$. Unfortunately, we cannot just work with the same cut-off function for all higher derivatives of $G$. This is because when we perform integration by parts, we end up with derivatives of this cut-off function and thus we need to use a hierarchy of these functions, depending on how many derivatives we take of \eref{lin_eq_visc_landau}. We also note that the terms involving the derivatives of the cut-off function can be thought of as the boundary error terms and only show up because we work on $\Omega_R$ and not on $\R^6$.

Let $\psi$ be a smooth, radial, nonnegative cutoff function in the velocity variable such that it is identically $1$ when $|v|\leq 1$ and vanishes for $|v|\geq 11/10$. For $0<r<R$ then define $\psi^r(v)=\psi(v/r)$. Finally, for $|\alpha|+|\beta|=m$, we define $\psi_m(v)=\psi^{R/2^m}(v)$. Note that for $m=0$ this means that $\psi_0\equiv 1$ on $\Omega_R$ and for any $m>0$ we have that $\psi_m$ vanishes at $\part \Omega_R$.\\
With the above definitions and lemma we can now have the following energy estimate for the solution of \eref{lin_eq_visc_landau} and its higher derivatives.
\begin{lemma}\label{l.energy_est_et_up}
Let $G$ be a solution to \eref{lin_diff_visc_eq} then for $|\alpha|+|\beta|\leq 10$ and $T\in [0,T_0)$ we have the following energy estimate
\begin{multline}\label{e.main_energy_estimate}
\norm{\jap{v}^{\omega_{\alpha,\beta}}(\der G) \psi_m}_{L^2_vL^2_x}^2(T)+\kappa \norm{\jap{v}^{\omega_{\alpha,\beta}}\jap{v}^{\frac{1}{2}}(\der G) \psi_m}_{L^2([0,T];L^2_vL^2_x)}^2\\
+\eps\norm{\jap{v}^{\omega_{\alpha,\beta}}\part_{x}(\der G) \psi_m}_{L^2([0,T];L^2_vL^2_x)}^2+\eps\norm{\jap{v}^{\omega_{\alpha,\beta}}\part_{v}(\der G) \psi_m}_{L^2([0,T];L^2_vL^2_x)}^2,\\
\lesssim \norm{\jap{v}^{\omega_{\alpha,\beta}}(\der g_{\ini}) \psi_m}_{L^2_vL^2_x}^2+\int_0^T \int \int \jap{v}^{2\omega_{\alpha,\beta}}(\der G)J(t,x,v) \psi_m\d v \d x\d t \\
+\mathcal{A}^{\alpha,\beta}_1+\mathcal{A}^{\alpha,\beta}_2+\mathcal{A}^{\alpha,\beta}_3
+\mathcal{B}^{\alpha,\beta}_{1}+\mathcal{B}^{\alpha,\beta}_{2}+\mathcal{B}^{\alpha,\beta}_{3}+\mathcal{B}^{\alpha,\beta}_{4}.
\end{multline}
Here $J$ are Terms $1$-Terms $6$ and are a mix of bulk and boundary error terms and are estimated in \lref{J_term} and
\begin{equation}\label{e.a1}
\mathcal{A}_1^{\alpha,\beta}:=\max_{i,j}\norm{\jap{v}^{2\omega_{\alpha,\beta}}\psi_m^2(\part^2_{v_iv_j}\a[h])(\der G)^2}_{L^1([0,T];L^1_vL^1_x)},
\end{equation}
\begin{equation}\label{e.a2}
\mathcal{A}_2^{\alpha,\beta}:=\max_{i,j}\norm{\jap{v}^{2\omega_{\alpha,\beta}-2}\psi_m^2\a[h](\der G)^2}_{L^1([0,T];L^1_vL^1_x)},
\end{equation}
\begin{equation}\label{e.a3}
\mathcal{A}_3^{\alpha,\beta}:=\max_{i,j,l}\norm{\jap{v}^{2\omega_{\alpha,\beta}-1}\psi_m^2(\part_{v_l}\a[h])(\der G)^2}_{L^1([0,T];L^1_vL^1_x)},
\end{equation}
\begin{equation}\label{e.b1}
\mathcal{B}_1^{\alpha,\beta}:=\max_{i,j}\norm{\jap{v}^{2\omega_{\alpha,\beta}}\psi_m \part^2_{v_iv_j}\psi_m(\a[h]+\eps\delta_{ij})(\der G)^2}_{L^1([0,T];L^1_vL^1_x)},
\end{equation}
\begin{equation}\label{e.b2}
\mathcal{B}_2^{\alpha,\beta}:=\max_{i,j,k,l}\norm{\jap{v}^{2\omega_{\alpha,\beta}}\part_{v_l}\psi_m \part_{v_k}\psi_m(\a[h]+\eps\delta_{ij})(\der G)^2}_{L^1([0,T];L^1_vL^1_x)},
\end{equation}
\begin{equation}\label{e.b3}
\mathcal{B}_3^{\alpha,\beta}:=\max_{i,j,l}\norm{\jap{v}^{2\omega_{\alpha,\beta}-1}\psi_m \part_{v_l}\psi_m(\a[h]+\eps\delta_{ij})(\der G)^2}_{L^1([0,T];L^1_vL^1_x)},
\end{equation}
\begin{equation}\label{e.b4}
\mathcal{B}_4^{\alpha,\beta}:=\max_{i,j,l}\sum_{\substack{|\alpha'''|+|\alpha''|+|\alpha'|=2|\alpha|\\ |\beta'''|+|\beta''|+|\beta'|= 2|\beta|+1\\ |\alpha''|+|\beta''|=|\alpha|+|\beta|\\ |\beta'|+|\alpha'|=1}}\norm{\jap{v}^{2\omega_{\alpha,\beta}}\psi_m \part_{v_l}\psi_m\derv{'''}{'''}G(\derv{'}{'}\a[h])\derv{''}{''} G}_{L^1([0,T];L^1_vL^1_x)}.
\end{equation}
\end{lemma}
\begin{proof}
We first multiply \eref{lin_diff_visc_eq} by $\jap{v}^{2\omega_{\alpha,\beta}}(\der G) \psi_m^2$ and integrate in $[0,T]\times \R^3\times \R^3$ (we suppress the domain of integration henceforth). Although, $G$ is only defined on $\Omega_R$, we multiply by the cut-off function $\psi_m$, thus we can extend domain of integration to all of $\R^6$ by assuming the integrand is zero outside $\Omega_R$.

Then integrating by parts in space and time we get,
\begin{align}
&\frac{1}{2}\int \int \jap{v}^{2\omega_{\alpha,\beta}}\psi_m^2(\der G)^2(T,x,v)\d v\d x-\frac{1}{2}\int \int \jap{v}^{2\omega_{\alpha,\beta}}\psi_m^2(\der G)^2(0,x,v) \d v \d x\\
&+\int_0^T \int \int \jap{v}^{2\omega_{\alpha,\beta}}\psi_m^2\kappa \jap{v} (\der G)^2(t,x,v)\d v \d x \d t\\\label{e.a_matrix_term}
 &-\int_0^T \int \int \jap{v}^{2\omega_{\alpha,\beta}}\psi_m^2\der G\a[h]\part^2_{v_i v_j}\der G\d v\d x\d t\\\label{e.lap_x_term}
&-\eps \int_0^T \int \int \jap{v}^{2\omega_{\alpha,\beta}}\psi_m^2 (\der G) \part^2_{x_i,x_i} \der G \d v \d x \d t\\\label{e.lap_v_term}
&-\eps \int_0^T \int \int \jap{v}^{2\omega_{\alpha,\beta}}\psi_m^2 (\der G) \part^2_{v_i,v_i} \der G \d v \d x \d t\\\label{e.H_term}
&=\int_0^T \int \int \jap{v}^{2\omega_{\alpha,\beta}}\psi^2_m(\der G)J(t,x,v)\d v \d x\d t.
\end{align}
Integrating by parts twice in $v$ for \eref{a_matrix_term} we get (for brevity we drop the integration signs)
\begin{equation*}
\begin{split}
\eref{a_matrix_term}&\equiv \jap{v}^{2\omega_{\alpha,\beta}}\psi_m^2 \part_{v_i}\der G(\a[h])\part_{v_j}\der G+\frac{1}{2}\jap{v}^{2\omega_{\alpha,\beta}}\psi_m^2 \part_{v_i}\a[h]\part_{v_j}((\der G)^2)\\
&\quad+\omega_{\alpha,\beta}v_i \jap{v}^{2\omega_{\alpha,\beta}-2}\psi_m^2\a[h]\part_{v_j}((\der G)^2)+\jap{v}^{2\omega_{\alpha,\beta}}\psi_m\part_{v_i}\psi_m\a[h]\part_{v_j}((\der G)^2)\\
&\equiv \underbrace{\jap{v}^{2\omega_{\alpha,\beta}}\psi_m^2 \part_{v_i}\der G(\a[h])\part_{v_j}\der G}_{A_1}- \underbrace{\frac{1}{2}\jap{v}^{2\omega_{\alpha,\beta}}\psi_m^2 \part_{v_iv_j}^2\a[h](\der G)^2}_{A_2}\\
&\quad- \underbrace{\omega_{\alpha,\beta}v_j\jap{v}^{2\omega_{\alpha,\beta}-2}\psi_m^2 \part_{v_i}\a[h](\der G)^2}_{A_3}- \underbrace{\jap{v}^{2\omega_{\alpha,\beta}}\psi_m\part_{v_j}\psi_m \part_{v_i}\a[h](\der G)^2}_{A_4}\\
&\quad - \underbrace{\omega_{\alpha,\beta}\delta_{i,j} \jap{v}^{2\omega_{\alpha,\beta}-2}\psi_m^2\a[h](\der G)^2}_{A_5}- \underbrace{2\omega_{\alpha,\beta}(\omega_{\alpha,\beta}-1)v_i v_j\jap{v}^{2\omega_{\alpha,\beta}-4}\psi_m^2\a[h](\der G)^2}_{A_6}\\
&\quad - \underbrace{\omega_{\alpha,\beta}v_i \jap{v}^{2\omega_{\alpha,\beta}-2}\psi_m\part_{v_j}\psi_m (\a[h])(\der G)^2}_{A_7}- \underbrace{\omega_{\alpha,\beta}v_i \jap{v}^{2\omega_{\alpha,\beta}-2}\psi_m^2\part_{v_j}\a[h](\der G)^2}_{A_8}\\
&\quad- \underbrace{2\omega_{\alpha,\beta}v_j\jap{v}^{2\omega_{\alpha,\beta}-2}\psi_m\part_{v_i}\psi_m(\a[h])(\der G)^2}_{A_9}- \underbrace{\jap{v}^{2\omega_{\alpha,\beta}}\part_{v_j}\psi_m\part_{v_i}\psi_m(\a[h])(\der G)^2}_{A_{10}}\\
&\quad- \underbrace{\jap{v}^{2\omega_{\alpha,\beta}}\psi_m\part_{v_jv_i}^2\psi_m(\a[h])(\der G)^2}_{A_{11}}- \underbrace{\jap{v}^{2\omega_{\alpha,\beta}}\psi_m\part_{v_i}\psi_m\part_{v_j}\a[h](\der G)^2}_{A_{12}}.
\end{split}
\end{equation*}
For now, we just make the following observations,
\begin{itemize} 
\item $A_1$ is non-negative and thus can be dropped from our estimate. This is true since $h\geq 0$.\\
\item $A_2$ is bounded by $\mathcal{A}_1^{\alpha,\beta}$.\\
\item $A_3$ and $A_8$ can be bounded by $\mathcal{A}_3^{\alpha,\beta}$ from \lref{energy_est_et_up}.\\
\item $A_5$ and $A_6$ can be bounded by $\mathcal{A}_2^{\alpha,\beta}$ from \lref{energy_est_et_up}.\\
\item $A_4$ and $A_{12}$ can be bounded by $\mathcal{B}_4^{\alpha,\beta}$ from \lref{energy_est_et_up}.\\
\item $A_7$ and $A_9$ can be bounded by $\mathcal{B}_3^{\alpha,\beta}$ from \lref{energy_est_et_up}.\\
\item $A_{10}$ is bounded by $\mathcal{B}_2^{\alpha,\beta}$ from \lref{energy_est_et_up}.\\
\item Finally $A_{11}$ is bounded by $\mathcal{B}_1^{\alpha,\beta}$.
\end{itemize}

All the terms other than $A_1$ will be treated as either ``bulk" error terms (when there are no derivatives falling on $\psi_m$) or as ``boundary" error terms  (when there are some derivatives falling on $\psi_m$).

Continuing integrating by parts we perform the same in $x$ for \eref{lap_x_term} to get $$\eref{lap_x_term}=\eps\int_0^T \int \int \jap{v}^{2\omega_{\alpha,\beta}}\psi^2_m(\part_{x_i}\der G)^2\d v \d x\d t.$$
Integrating by parts in $v$ for \eref{lap_v_term} we get 
\begin{equation*}
\begin{split}
\eref{lap_v_term}&\equiv \eps [\jap{v}^{2\omega_{\alpha,\beta}}\psi^2_m (\part_{v_i}\der G)^2+ \jap{v}^{2\omega_{\alpha,\beta}}\part_{v_i}((\der G)^2) \psi_m \part_{v_i}\psi_m\\
&\quad+\omega_{\alpha,\beta}v_i\jap{v}^{2\omega_{\alpha,\beta}-2}\part_{v_i}((\der G)^2) \psi_m^2]\\
&\equiv \eps[ \underbrace{\jap{v}^{2\omega_{\alpha,\beta}}\psi^2_m (\part_{v_i}\der G)^2}_{C_1}-\underbrace{2\omega_{\alpha,\beta}v_i\jap{v}^{2\omega_{\alpha,\beta}-2}(\der G)^2 \psi_m\part_{v_i}\psi_m}_{C_2}\\
&\quad-\underbrace{\jap{v}^{2\omega_{\alpha,\beta}}(\der G)^2(\part_{v_i} \psi_m)^2}_{C_3}-\underbrace{\jap{v}^{2\omega_{\alpha,\beta}}(\der G)^2\psi_m\part^2_{v_i} \psi_m}_{C_4}\\
&\quad-\underbrace{\omega_{\alpha,\beta}\jap{v}^{2\omega_{\alpha,\beta}-2}(\der G)^2 \psi_m^2}_{C_5}-\underbrace{2\omega_{\alpha,\beta}(\omega_{\alpha,\beta}-1)v_i^2\jap{v}^{2\omega_{\alpha,\beta}-4}(\der G)^2\psi_m^2}_{C_6}\\
&\quad-\underbrace{\omega_{\alpha,\beta}v_i\jap{v}^{2\omega_{\alpha,\beta}-2}(\der G)^2 \psi_m\part_{v_i}\psi_m}_{C_7}].
\end{split}
\end{equation*}

We again bound the various $C_i$ terms as follows,
\begin{itemize}
\item $C_1$ is positive and will be useful for closing the estimates and hence incorporated on the LHS.\\
\item $C_2$ and $C_7$ can be bounded by $\mathcal{B}_{3}^{\alpha,\beta}$\\
\item $C_3$ can be bounded by $\mathcal{B}_2^{\alpha,\beta}$\\
\item $C_4$ can be bounded by $\mathcal{B}_1^{\alpha,\beta}$\\
\item $C_5$ and $C_6$ can be absorbed on the LHS by choosing $\eps$ small.
\end{itemize}
\end{proof} 
\begin{remark}\label{r.psi_0=1}
For $|\alpha|+|\beta|=0$, we have by our convention $\psi_0=1$. Thus the terms of the form $\mathcal{B}_i^{\alpha,\beta}$ are not present. This fact will be crucial later in showing that the boundary error terms degenerate in the limit $R \to \infty$ via an induction.
\end{remark}
Now we bound $\int_0^T \int \int \jap{v}^{2\omega_{\alpha,\beta}}(\der G)J(t,x,v) \psi_m^2\d v \d x\d t $ with various error terms.

\begin{lemma}\label{l.J_term}
For $|\alpha|+|\beta|\leq 10$, we have that\\
\begin{align*}
 \int_0^T \int \int \jap{v}^{2\omega_{\alpha,\beta}}(\der G)J(t,x,v)\psi_m\d v \d x\d t &\lesssim T_1^{\alpha,\beta}+T_2^{\alpha,\beta}+T^{\alpha,\beta}_{3,1}+T^{\alpha,\beta}_{3,2}+T^{\alpha,\beta}_{3,3}
+T^{\alpha,\beta}_{4}\\
&\quad+T^{\alpha,\beta}_{5,1}+T^{\alpha,\beta}_{5,2}+T^{\alpha,\beta}_{6,1}+T^{\alpha,\beta}_{6,2}+\mathcal{B}_4^{\alpha,\beta}+\mathcal{B}_5^{\alpha,\beta},
\end{align*}
with $\mathcal{B}_4^{\alpha,\beta}$ defined in \eref{b4},\\
\begin{equation}\label{e.b5}
\mathcal{B}_5^{\alpha,\beta}:=\max_{i,j,l}\norm{\jap{v}^{2\omega_{\alpha,\beta}}\psi_m \part_{v_l}\psi_m\left((\a[h]+\eps\delta_{ij}) \frac{v_i}{\jap{v}}\right)(\der G)^2}_{L^1([0,T];L^1_vL^1_x)}.
\end{equation}
Moreover,
\begin{equation}\label{e.T1}
T^{\alpha,\beta}_1:=\sum_{\substack{|\alpha'|\leq |\alpha|+1\\ |\beta'|\leq |\beta|-1}}\norm{\jap{v}^{2\omega_{\alpha,\beta}}\psi_m^2|\der G|\cdot|\derv{'}{'}G|}_{L^1([0,T];L^1_xL^1_v)},
\end{equation}
\begin{equation}\label{e.T2}
T_2^{\alpha,\beta}:=\sum \limits_{|\beta'|\leq |\beta|-1}\norm{\kappa\jap{v}^{2\omega_{\alpha,\beta}}\psi_m^2|\der G|\cdot|\derv{}{'} G|}_{L^1([0,T];L^1_xL^1_v)}
\end{equation}
\begin{equation}\label{e.T3_1}
T_{3,1}^{\alpha,\beta}:=\max_{i,j} \sum_{\substack{|\alpha'|+|\alpha''|+|\alpha'''|\leq 2|\alpha|\\ |\beta'|+|\beta''|+|\beta'''|\leq 2|\beta|+2\\ |\alpha'''|+|\beta'''|=|\alpha|+|\beta|\\ 2\leq |\alpha'|+|\beta'|\leq \min\{|\alpha|+|\beta|,8\}\\ |\alpha''|+|\beta''|\leq |\alpha|+|\beta|}}\norm{\jap{v}^{2\omega_{\alpha,\beta}}\psi_m^2|\derv{'''}{'''} G||\derv{'}{'}\a[h]||\derv{''}{''}G|}_{L^1([0,T];L^1_xL^1_v)},
\end{equation}
\begin{equation}\label{e.T3_2}
T_{3,2}^{\alpha,\beta}:=\max_{i,j} \sum_{\substack{|\alpha'|+|\alpha''|= |\alpha|\\ |\beta'|+|\beta''|=|\beta|\\ |\alpha'|+|\beta'|\geq 9}}\norm{\jap{v}^{2\omega_{\alpha,\beta}}\psi_m^2|\der G||\derv{'}{'}\a[h]||\part^2_{v_iv_j}\derv{''}{''}G|}_{L^1([0,T];L^1_xL^1_v)},
\end{equation}
\begin{equation}\label{e.T3_3}
T_{3,3}^{\alpha,\beta}:=\max_{i,j,l} \sum_{\substack{|\alpha'|+|\alpha''|= |\alpha|\\ |\beta'|+|\beta''|=\beta|\\  |\alpha'|+|\beta'|=1}}\norm{\jap{v}^{2\omega_{\alpha,\beta}-1}\psi_m^2|\der G||\derv{'}{'}\a[h]|\part_{v_l}\derv{''}{''}G|}_{L^1([0,T];L^1_xL^1_v)},
\end{equation}
\begin{equation}\label{e.T4}
T_4^{\alpha,\beta}:=\sum_{\substack{|\alpha'|+|\alpha''|\leq |\alpha|\\ |\beta'|+|\beta''|\leq |\beta|}} \norm{\jap{v}^{2\omega_{\alpha,\beta}}\psi_m^2|\der G||\derv{'}{'}\cm[h]||\derv{''}{''}G|}_{L^1([0,T];L^1_xL^1_v)},
\end{equation}
\begin{equation}\label{e.T5_1}
T_{5,1}^{\alpha,\beta}:=\sum_{\substack{|\alpha'|+|\alpha''|=|\alpha|\\ |\beta'|+|\beta''|=|\beta|+1\\ 1\leq |\alpha'|+|\beta'|\leq |\alpha|+|\beta|}}\norm{\jap{v}^{2\omega_{\alpha,\beta}}\psi_m^2|\der G|\left|\derv{'}{'}\left((\a[h]+\eps\delta_{ij}) \frac{v_i}{\jap{v}}\right)\right||\derv{''}{''}G|}_{L^1([0,T];L^1_xL^1_v)},
\end{equation}
\begin{equation}\label{e.T5_2}
T_{5,2}^{\alpha,\beta}:=\max_{j}\norm{\jap{v}^{2\omega_{\alpha,\beta}-1}\psi_m^2|\der G|\left|\left((\a[h]+\eps\delta_{ij})\frac{v_i}{\jap{v}}\right)\right||\der G|}_{L^1([0,T];L^1_xL^1_v)},
\end{equation}
\begin{equation}\label{e.T6_1}
T_{6,1}^{\alpha,\beta}:=\sum_{\substack{|\alpha'|+|\alpha''|\leq |\alpha|\\ |\beta'|+|\beta''|\leq |\beta|\\ |\alpha'|+|\beta'|\geq 1}} \norm{\jap{v}^{2\omega_{\alpha,\beta}}\psi_m^2|\der G|\left|\derv{'}{'}\left(\frac{\bar{a}[h]_{ii}+\eps}{\jap{v}}\right)\right||\derv{''}{''} G|}_{L^1([0,T];L^1_xL^1_v)},
\end{equation}
\begin{equation}\label{e.T6_2}
T_{6,2}^{\alpha,\beta}:=\sum_{\substack{|\alpha'|+|\alpha''|\leq |\alpha|\\ |\beta'|+|\beta''|\leq |\beta|}} \norm{\jap{v}^{2\omega_{\alpha,\beta}}\psi_m^2|\der G|\left|\derv{'}{'}\left(\frac{(\a[h]+\eps\delta_{ij})v_iv_j}{\jap{v}^2} \right)\right||\derv{''}{''} G|}_{L^1([0,T];L^1_xL^1_v)}.
\end{equation}
\end{lemma}
\begin{proof}
We proceed by bounding the energy of the various terms appearing on the RHS of \eref{lin_diff_visc_eq}.

\emph{Term 1:} We have that $$[\part_t+v_i\part_{x_i},\der]G=\sum_{\substack{|\beta'|+|\beta''|=|\beta|\\ |\beta'|=1}}\part^{\beta'}_{x}\part^{\beta''}_{v}\part^\alpha_x G$$\\
Thus it can be bounded by $T^{\alpha,\beta}_1$.

\emph{Term 2:}  For Term $2$, the commutator term arises from $\part_v$ acting on $\jap{v}$. Thus, we have
$$|\text{Term }2|\lesssim \sum \limits_{|\beta'|\leq |\beta|-1} \derv{}{'} G$$
This contribution is majorized by $T_2^{\alpha,\beta}$.

\emph{Term 3:} When only one derivative hits $\a[h]$, we need to perform integration by parts. We split this into two cases

\emph{Case 1}: When $(\alpha',\beta')=(1,0)$. That is $\part_x^{\alpha'}=\part_{x_l}$.\\
We perform integration by parts twice first with $\part_{v_i}$ then followed by $\part_{x_l}$.
\begin{align}
\jap{v}^{2\omega_{\alpha,\beta}}\psi_m^2\der G &(\part_{x_l}\a[h])\part^2_{v_iv_j}\derv{''}{} G \nonumber\\
&\equiv -\jap{v}^{2\omega_{\alpha,\beta}} \psi_m^2\part_{v_i}\part_{x_l}\derv{''}{} G(\part_{x_l}\a[h])\part_{v_j}\derv{''}{}G \nonumber\\
&\quad -2\jap{v}^{2\omega_{\alpha,\beta}}\psi_m\part_{v_i}\psi_m\der G (\part_{x_l}\a[h])\part_{v_j}\derv{''}{}G \nonumber\\
&\quad-2(\omega_{\alpha,\beta})v_i\jap{v}^{2\omega_{\alpha,\beta}-2}\psi_m^2\der G (\part_{x_l}\a[h])\part_{v_j}\derv{''}{}G \nonumber\\
&\quad-\jap{v}^{2\omega_{\alpha,\beta}}\psi_m^2\der G(\part_{v_i}\part_{x_l}\a[h])\part_{v_j}\derv{''}{} G \nonumber\\
&\equiv \frac{1}{2}\jap{v}^{2\omega_{\alpha,\beta}}\psi_m^2\part_{v_i}\derv{''}{}g(\part_{x_l}^2 \a[h])\part_{v_j}\derv{''}{} g\label{e.x_der_a_1}\\
&\quad -2\jap{v}^{2\omega_{\alpha,\beta}}\psi_m\part_{v_i}\psi_m\der G (\part_{x_l}\a[h])\part_{v_j}\derv{''}{}G \label{e.x_der_a_2}\\
&\quad-2(\omega_{\alpha,\beta})v_i\jap{v}^{2\omega_{\alpha,\beta}-2}\psi_m^2\der G (\part_{x_l}\a[h])\part_{v_j}\derv{''}{}G \label{e.x_der_a_3}\\
&\quad-\jap{v}^{2\omega_{\alpha,\beta}}\psi_m^2\der G(\part_{v_i}\part_{x_l}\a[h])\part_{v_j}\derv{''}{} G \label{e.x_der_a_4}.
\end{align}

By our definitions we have that $|\eref{x_der_a_4}|+|\eref{x_der_a_1}|\lesssim T^{\alpha,\beta}_{3,1}$, $|\eref{x_der_a_2}|\lesssim \mathcal{B}_4^{\alpha,\beta}$, $|\eref{x_der_a_3}|\lesssim T^{\alpha,\beta}_{3,3}$.

\emph{Case 2}:  When $(\alpha',\beta')=(0,1)$. That is $\part_v^{\beta'}=\part_{v_l}$.\\
We perform integration by parts twice first with $\part_{v_i}$ then followed by $\part_{v_l}$.
\begin{align}
\jap{v}^{2\omega_{\alpha,\beta}}\psi^2_m\der G& (\part_{v_l}\a[h])\part^2_{v_iv_j}\derv{}{''} G \nonumber\\
&\equiv -\jap{v}^{2\omega_{\alpha,\beta}}\psi^2_m \part_{v_i}\part_{v_l}\derv{}{''} G(\part_{v_l}\a[h])\part_{v_j}\derv{}{''}G \nonumber\\
& \quad-2\jap{v}^{2\omega_{\alpha,\beta}}\psi_m\part_{v_i}\psi_m\der G (\part_{v_l}\a[h])\part_{v_j}\derv{}{''}G \nonumber\\
&\quad-2(\omega_{\alpha,\beta})v_i\jap{v}^{2\omega_{\alpha,\beta}-2}\psi_m^2\der G (\part_{v_l}\a[h])\part_{v_j}\derv{}{''}G \nonumber\\
&\quad-\jap{v}^{2\omega_{\alpha,\beta}}\psi_m^2\der G(\part_{v_i}\part_{v_l}\a[h])\part_{v_j}\derv{}{''} G \nonumber\\
&\equiv \frac{1}{2}\jap{v}^{2\omega_{\alpha,\beta}}\psi_m^2\part_{v_i}\derv{}{''}g(\part_{v_l}^2 \a[h])\part_{v_j}\derv{}{''} G \label{e.v_der_a_1}\\
&\quad+\jap{v}^{2\omega_{\alpha,\beta}}\psi_m\part_{v_l}\psi_m\part_{v_i}\derv{}{''} G (\part_{v_l}\a[h])\part_{v_j}\derv{}{''}G \label{e.v_der_a_2}\\
&\quad+(\omega_{\alpha,\beta})v_l\jap{v}^{2\omega_{\alpha,\beta}-2}\psi^2_m\derv{}{''} G (\part_{v_l}\a[h])\part_{v_j}\derv{}{''}G \label{e.v_der_a_3}\\
&\quad -2\jap{v}^{2\omega_{\alpha,\beta}}\psi_m\part_{v_i}\psi_m\der G (\part_{v_l}\a[h])\part_{v_j}\derv{}{''}G \label{e.v_der_a_4}\\
&\quad-2(\omega_{\alpha,\beta})v_i\jap{v}^{2\omega_{\alpha,\beta}-2}\psi_m^2\der G (\part_{v_l}\a[h])\part_{v_j}\derv{}{''}G \label{e.v_der_a_5}\\
&\quad-\jap{v}^{2\omega_{\alpha,\beta}}\psi_m^2\der G(\part_{v_i}\part_{v_l}\a[h])\part_{v_j}\derv{}{''} G \label{e.v_der_a_6}.
\end{align}

Again, we have that $|\eref{v_der_a_4}|+|\eref{v_der_a_1}|\lesssim T^{\alpha,\beta}_{3,1}$, $|\eref{v_der_a_2}|+|\eref{v_der_a_4}|\lesssim \mathcal{B}_4^{\alpha,\beta}$, $|\eref{v_der_a_3}|+|\eref{v_der_a_5}|\lesssim T^{\alpha,\beta}_{3,3}$.\\
When we have more than $1$ derivative hitting $\a[h]$ then we don't need to perform integration by parts. Since we treat the two cases- when less than $8$ derivatives hit $\a[h]$ and when at least $9$ derivatives hit $\a[h]$- differently, we use different error terms to bound them.

Explicitly, when $|\alpha'|+|\beta'|\leq 8$ we have $|\text{Term } 3|\lesssim T_{3,1}^{\alpha,\beta}$
 and when $|\alpha'|+|\beta'|\geq 9$ we have $|\text{Term } 3|\lesssim T_{3,2}^{\alpha,\beta}$.

Thus in total we have,
$$|\text{Term }3|\lesssim T^{\alpha,\beta}_{3,1}+T^{\alpha,\beta}_{3,2}+T^{\alpha,\beta}_{3,3}+\mathcal{B}^{\alpha,\beta}_{4}.$$
Note that depending on how many derivatives hit $\a[h]$ some of these terms might be zero.

\emph{Term 4:} This term can be bounded in a straight forward way by $T^{\alpha,\beta}_4$.

\emph{Term 5:} Again, if we have no derivatives falling on $\a[h]\frac{v_i}{\jap{v}}$ then we need to do integration by parts. More precisely, we have
\begin{equation}\label{e.Term_5}
\begin{split}
\part^\alpha_x \part_v^\beta \left((\a[h]+\eps\delta_{ij}) \frac{v_i}{\jap{v}}\part_{v_j}G\right)&=\sum_{\substack{|\alpha'|+|\alpha''|=|\alpha|\\ |\beta'|+|\beta''|=|\beta|\\ |\alpha'|+|\beta'|\geq 1}}\left( \derv{'}{'}\left((\a[h]+\eps\delta_{ij}) \frac{v_i}{\jap{v}}\right)\right)(\part_{v_j}\derv{''}{''}G)\\
&\quad+(\a[h]+\eps\delta_{ij}) \frac{v_i}{\jap{v}}\part_{v_j}\der G.
\end{split}
\end{equation}
For the second term we use integration by parts in $\part_{v_j}$ to get,
\begin{align}
2(d(t))\jap{v}^{2\omega_{\alpha,\beta}}\psi_m^2\der G&(\a[h]+\eps\delta_{ij})\frac{v_i}{\jap{v}}\part_{v_j}\der G \nonumber\\
&\equiv -(d(t)) \jap{v}^{2\omega_{\alpha,\beta}}\psi_m^2\left(\part_{v_j}\left((\a[h]+\eps\delta_{ij})\frac{v_i}{\jap{v}}\right)\right)(\der G)^2 \label{e.a_cont_v_1}\\
&\quad-2d(t)\omega_{\alpha,\beta}\jap{v}^{2\omega_{\alpha,\beta}-2}v_j\psi_m^2\left((\a[h]+\eps\delta_{ij}) \frac{v_i}{\jap{v}}\right)(\der G)^2 \label{e.a_cont_v_2}\\
&\quad-2d(t)\jap{v}^{2\omega_{\alpha,\beta}}\psi_m\part_{v_j}\psi_m\left((\a[h]+\eps\delta_{ij}) \frac{v_i}{\jap{v}}\right)(\der G)^2 \label{e.a_cont_v_3}.
\end{align}
The first term in \eref{Term_5} is bounded by $T^{\alpha,\beta}_{5,1}$. We also have,
$$|\eref{a_cont_v_1}|\lesssim T^{\alpha,\beta}_{5,1}, |\eref{a_cont_v_2}|\lesssim  T^{\alpha,\beta}_{5,2} \text{ and } |\eref{a_cont_v_3}|\lesssim \mathcal{B}^{\alpha,\beta}_5$$
 
\emph{Term 6:} For first part of Term 6 note that $\delta_{ij} \a[h]=\bar{a}_{ii}[h]$ and when no derivatives hit $\frac{\bar{a}_{ii}[h]}{\jap{v}}$ then we have the term $$-d(t)\int_0^T\int \int \jap{v}^{2\omega_{\alpha,\beta}}\psi_m^2(\der G)^2 \frac{\bar{a}_{ii}[h]+\eps}{\jap{v}}\d v \d x\d t.$$ Note, crucially that since $h\geq 0$, we have that $\frac{\bar{a}_{ii}[h]}{\jap{v}}\geq 0$ implying that the whole integral is negative and thus can be dropped.\\
When at least one derivative fall on $\frac{\bar{a}_{ii}[h]+\eps}{\jap{v}}$ then we can bound it by $T^{\alpha,\beta}_{6,1}$.

Similarly the second part of Term 6 can be bounded by $T^{\alpha,\beta}_{6,2}$.
\end{proof}
\begin{remark}
It is seen easily that $\mathcal{A}^{\alpha,\beta}_1\leq T^{\alpha,\beta}_{3,1}$ and $\mathcal{A}^{\alpha,\beta}_3\leq T^{\alpha,\beta}_{3,3}$. Thus it suffices to bound the $T^{\alpha,\beta}_i$, $\mathcal{A}^{\alpha,\beta}_2$ and $\mathcal{B}^{\alpha,\beta}_i$ terms.
\end{remark}
\section{Estimates for the coefficients}\label{s.coefficient_bounds}
In the next section we estimate all the errors but before we can do that we need to get a bound on the coefficient matrices, $\bar{a}$ and $\bar{c}$, and their derivatives. We thus begin with bounds for $\bar{a}$ and its derivatives.
\begin{proposition}\label{p.pointwise_estimates_a}
We assume $|\alpha|+|\beta|\leq 10$.\\
The coefficient $\a$ and its higher derivatives satisfy the following pointwise bounds:
\begin{equation}\label{e.pw_bound_a}
\max_{i,j}|\derv{}{} \a|(t,x,v)\lesssim \int |v-v_*|^{2+\gamma}|\der f|(t,x,v_*)\d v_*,
\end{equation}
\begin{equation}\label{e.pw_bound_a_v}
\max_{j}\left|\der \left(\a \frac{v_i}{\jap{v}}\right)\right|\lesssim \jap{v}^{1+\gamma} \int \jap{v_*}^{2+\gamma}|\der f|(t,x,v_*)\d v_*,
\end{equation}
\begin{equation}\label{e.pw_bound_a_2v}
\left|\der\left(\a \frac{v_iv_j}{\jap{v}^2}\right)\right|(t,x,v)\lesssim \jap{v}^\gamma\int\jap{v_*}^4 |\der f|(t,x,v_*)\d v_* .
\end{equation}

The first v-derivatives of $\a$ and the corresponding higher derivatives satisfy:
\begin{equation}\label{e.pw_bound_der_a}
\max{i,j,k}\left|\der \part_{v_k}\a\right|(t,x,v)\lesssim \int |v-v_*|^{1+\gamma}|\der f|(t,x,v_*)\d v_*,
\end{equation}
\begin{equation}\label{e.pw_bound_der_a_v}
\max{i,j,k}|\der \part_{v_k}\left(\a(t,x,v)\frac{v_i}{\jap{v}}\right)|\lesssim \jap{v}^\gamma\int \jap{v_*}^{2+\gamma}|\der f|(t,x,v_*)\d v_*.
\end{equation}

Finally, the second derivatives of $\a$ and its higher derivatives follow:
\begin{equation}\label{e.pw_bound_2der_a}
\max_{i,j,k,l}|\der \partial^2_{v_kv_l}\a|(t,x,v)\lesssim \int |v-v_*|^\gamma |\der f|(t,x,v_*)\d v_*.
\end{equation}
\end{proposition}
\begin{proof}
We use the following facts for convolutions with $|\beta'|\leq 2$,
$$\der \part^{\beta'}_v \a[h]=\int (\part_v^{\beta'}a_{ij}(v-v_*))(\der h)(t,x,v_*) \d v_*$$
$$\der \part^{\beta'}_v \left(\a[h]\frac{v_i}{\jap{v}}\right)=\int (\part_v^{\beta'}\left(a_{ij}\frac{v_i}{\jap{v}}\right)(v-v_*))(\der h)(t,x,v_*) \d v_*$$
\emph{Proof of \eref{pw_bound_a}:} For this we just notice via the form of matrix $a$ in \eref{matrix_a} that $$|a_{ij}(v-v_*)|\leq |v-v_*|^{2+\gamma}.$$

\emph{Proof of \eref{pw_bound_a_v}:} For \eref{pw_bound_a_v} we use \eref{matrix_a} to get
\begin{equation}\label{e.int_pw_bound_a_v}
\begin{split}
a_{ij}(v-v_*)v_i&=|v-v_*|^{2+\gamma}\left(v_j-\frac{(v\cdot (v-v_*))(v-v_*)_j}{|v-v_*|^2}\right)\\
&=|v-v_*|^\gamma\left(v_j(v-v_*)_l(v-v_*)_l-v_l(v-v_*)_l (v-v_*)_j\right)\\
&=|v-v_*|^\gamma (v-v_*)_l(v_l(v_*)_j-v_j(v_*)_l).
\end{split}
\end{equation}
Thus we have using triangle inequality that $\left|a_{ij}(v-v_*)\frac{v_i}{\jap{v}}\right| \lesssim \jap{v}^{1+\gamma}\jap{v_*}^{2+\gamma}$.

\emph{Proof of \eref{pw_bound_a_2v}:} For \eref{pw_bound_a_2v} we again begin with \eref{matrix_a} and see using triangle inequality
\begin{equation}
\begin{split}
a_{ij}(v-v_*)v_iv_j&=|v-v_*|^{2+\gamma}\left(|v|^2-\frac{(v\cdot (v-v_*))^2}{|v-v_*|^2}\right) \\
&=|v-v_*|^{\gamma}(|v|^2(|v|^2+|v_*|^2-2(v\cdot v_*))-|v|^4+2|v|^2(v\cdot v_*)-(v\cdot v_*)^2)\\
&=|v-v_*|^\gamma(|v|^2|v_*|^2-(v\cdot v_*)^2)\\
&\lesssim |v-v_*|^\gamma|v|^2|v_*|^2\\
&\lesssim \jap{v}^{2+\gamma}\jap{v_*}^{4+\gamma}.
\end{split}
\end{equation}
Hence we have that $\left|\frac{a_{ij}v_iv_j}{\jap{v}^2}\right|\lesssim \jap{v}^\gamma\jap{v_*}^{4+\gamma}$.

\emph{Proof of \eref{pw_bound_der_a}:} Using homogeneity we have,
$$|\part_{v_k}a_{ij}|\lesssim |v-v_*|^{1+\gamma},$$ which implies \eref{pw_bound_der_a}.

\emph{Proof of \eref{pw_bound_der_a_v}:} To prove \eref{pw_bound_der_a_v}, we use \eref{int_pw_bound_a_v} to get
\begin{equation*}
\begin{split}
\part_{v_k}\left(a_{ij}(v-v_*)\frac{v_i}{\jap{v}}\right)&=\part_{v_k}\left(\frac{1}{\jap{v}}|v-v_*|^\gamma (v-v_*)_l(v_l(v_*)_j-v_j(v_*)_l)\right).
\end{split}
\end{equation*}

Using the product rule we get mutiple terms which we treat one by one now.

When $\partial_{v_k}$ hit $\frac{1}{\jap{v}}$ we get upto a constant $\frac{v_k}{\jap{v}^3}$, thus we have using triangle inequality $$\frac{v_k}{\jap{v}^3}|v-v_*|^\gamma (v-v_*)_l(v_l(v_*)_j-v_j(v_*)_l)\lesssim \jap{v}^\gamma\jap{v_*}^{2+\gamma}.$$

When $\partial_{v_k}$ hit $|v-v_*|^\gamma$ we get upto a constant $(v-v_*)_k|v-v_*|^{\gamma-2}$, hence using the fact that $\frac{(v-v_*)_k(v-v_*)l}{|v-v_*|^2}\lesssim 1$ and triangle inequality we get
$$\frac{1}{\jap{v}}|v-v_*|^\gamma\frac{(v-v_*)_k(v-v_*)l}{|v-v_*|^2}(v_l(v_*)_j-v_j(v_*)_l) \lesssim \jap{v}^\gamma\jap{v_*}^{\gamma+1}.$$

Finally when $\part_{v_k}$ hits $(v_l(v_*)_j-v_j(v_*)_l)$ we get $(\delta_{lk}(v_*)_j-\delta_{jk}(v_*)_l)$.

Again, by triangle inequality we get 
$$\frac{1}{\jap{v}}|v-v_*|^\gamma (v-v_*)_l(\delta_{lk}(v_*)_j-\delta_{jk}(v_*)_l)\lesssim \jap{v}^\gamma\jap{v_*}^{2+\gamma}.$$

Hence in total we have $$\part_{v_k}\left(a_{ij}(v-v_*)\frac{v_i}{\jap{v}}\right) \lesssim \jap{v}^{\gamma}\jap{v_*}^{2+\gamma}.$$

\emph{Proof of \eref{pw_bound_2der_a}:} Finally, for the second derivatives of $a_{ij}$ we obtain by homogeneity\\
$$|\part_{v_l}\part_{v_k}a_{ij}(v-v_*)|\lesssim |v-v_*|^\gamma.$$ which implies \eref{pw_bound_2der_a}
\end{proof}
Since $c=\part^2_{z_iz_j}a_{ij}(z)$, we see that $\cm$ and its higher derivatives satisfy the same bounds as \eref{pw_bound_2der_a}.

Now we note a very simple interpolation inequality that will let us estimate the coefficient matrices in $L^\infty_v$.\\
\begin{lemma}\label{l.L1_to_L2_v}
Let $h:[0,T_0)\times\R^3\times\R^3$ be a smooth function, then\\
$$\norm{h}_{L^1_v}(t,v)\lesssim \norm{\jap{v}^2h}_{L^2_v}(t,x).$$
\end{lemma}
\begin{proof}
We just use Holder's inequality to get\\
\begin{align*}
\int |h|(t,x,v)\d v&\lesssim (\int \jap{v}^4h^2(t,x,v)\d v)^{\frac{1}{2}}(\int \jap{v}^{-4}\d v)^{\frac{1}{2}}\\
&\lesssim  (\int \jap{v}^4h^2(t,x,v)\d v)^{\frac{1}{2}}.
\end{align*}
\end{proof}
As will become clear from \lref{special_holder_der_leq_8} and \lref{special_holder_der_geq_9}, we estimate the derivatives of $\a[h]$ and $\cm[h]$ in different $L^p_x$ spaces but always in $L^\infty_v$. Thus in the following lemma we establish an estimate in $L^\infty_v$.

But before that we note a simple bound on derivatives of $h$.
\begin{lemma}\label{l.exp_bound}
For every $l\in \N$, the following estimate holds with a constant depending on $l,\gamma$ and $d_0$ for any $(t,x,v) \in [0,T_0)\times \R^3\times \R^3$
$$\jap{v}^l|\der h|(t,x,v)\lesssim_l \sum \limits_{|\beta'|\leq |\beta|}|\derv{}{'} \h|(t,x,v),$$
where $\h(t,x,v)=he^{d(t)\jap{v}}$.
\end{lemma}
\begin{proof}
This is immediate from differentiating $\h$ and using that $\jap{v}^ne^{-d(t)\jap{v}}\lesssim_n 1$ for all $n\in \N$.
\end{proof}
\begin{lemma}\label{l.L_inf_a}
For $|\alpha|+|\beta|\leq 10$, we have the following $L^\infty_v$ bounds
\begin{equation}\label{e.L_inf_a}
\norm{\jap{v}^{-2-\gamma}\der \a[h]}_{L^\infty_v}(t,x) \lesssim \norm{\h}_{Y^{\alpha+\beta}_v}(t,x),
\end{equation}
\begin{equation}\label{e.L_inf_a_v}
\norm{\jap{v}^{-1-\gamma}\der \left(\a[h]\frac{v_i}{\jap{v}}\right)}_{L^\infty_v}(t,x) \lesssim \norm{\h}_{Y^{\alpha+\beta}_v}(t,x),
\end{equation}
\begin{equation}\label{e.L_inf_a_2_v}
\norm{\jap{v}^{-\gamma}\der \left(\a[h]\frac{v_iv_j}{\jap{v}^2}\right)}_{L^\infty_v}(t,x) \lesssim \norm{\h}_{Y^{\alpha+\beta}_v}(t,x).
\end{equation}

If $|\beta|\geq 1$ then we have,
\begin{equation}\label{e.L_inf_a_1_der}
\norm{\jap{v}^{-1-\gamma}\derv{}{'} \part_{v_l} \a[h]}_{L^\infty_v}(t,x) \lesssim\norm{\h}_{Y^{\alpha+\beta'}_v}(t,x),
\end{equation}
\begin{equation}\label{e.L_inf_a_v_1_der}
\norm{\jap{v}^{-\gamma}\derv{}{'} \part_{v_l} \left(\a[h]\frac{v_i}{\jap{v}}\right)}_{L^\infty_v}(t,x) \lesssim \norm{\h}_{Y^{\alpha+\beta'}_v}(t,x),
\end{equation}
where $\part_{v}^\beta=\part_v^{\beta'}\part_{v_l}$.

if $|\beta|\geq 2$ then we have,
\begin{equation}\label{e.L_inf_a_2_der}
\norm{\jap{v}^{-\gamma}\derv{}{'} \part^2_{v_lv_k} \a[h]}_{L^\infty_v}(t,x) \lesssim \norm{\h}_{Y^{\alpha+\beta'}_v}(t,x)(t,x),
\end{equation}
where $\part_{v}^\beta=\part_v^{\beta'}\part^2_{v_kv_l}.$
\end{lemma}
\begin{proof}
This is just a straightforward combination of appropriate bounds in \pref{pointwise_estimates_a} and \lref{L1_to_L2_v} which implies the bound of the form,
$$\text{LHS}\lesssim \norm{\jap{v}^k\der h}_{L^2_v}(t,x).$$
We finish of by using \lref{exp_bound}.
\end{proof}
In the following lemma we establish an $L^\infty_v$ bound for $\cm[h]$.
\begin{lemma}\label{l.L_inf_c}
For $|\alpha|+|\beta|\leq 10$, we have\\
\begin{equation}\label{e.L_inf_c}
\norm{\jap{v}^{-\gamma}\der \cm[h]}(t,x) \lesssim \norm{\h}_{Y_v^{\alpha+\beta}}(t,x),
\end{equation}
where, again, $\h=he^{d(t)\jap{v}}$.
\end{lemma}
\begin{proof}
Since we have\\
$$|\der \cm|(t,x,v)\lesssim \int|v-v_*|^\gamma |\der h|\d v_*,$$
the required result just follows by triangle inequality and in the same way as \lref{L_inf_a}.
\end{proof}
\begin{remark}
At the level of the linear equation we can bound $h$ in a much weaker norm but since we will perform an iteration, we stick to the strong norm.
\end{remark}
Having bounded the coefficient matrices we note two general but specialized inequalities which play a crucial role in bounding the error terms in the next section. In fact, we will reduce all the error terms to fit one of the following lemmas.
\begin{lemma}\label{l.special_holder_der_leq_8}
Let $G$ be a solution to \eref{lin_eq_visc_landau} and $\h=e^{d(t)\jap{v}}h$ where $h$ is the nonnegative function from \lref{lin_visc_landau_exis}. For $|\alpha''|+|\beta''|\leq 10$ and $|\alpha'''|+|\beta'''|\leq 10$ we have the following estimate,
\begin{multline*}
\norm{\jap{v}^{\omega_{\alpha'',\beta''}+\omega_{\alpha''',\beta'''}}|\derv{'''}{'''} G|\jap{v}^{\gamma} \norm{\h}_{Y^8_v}|\derv{''}{''} G|}_{L^1([0,T];L^2_xL^2_v)}\\ \lesssim \norm{\h}_{\H}[\norm{\jap{v}^{\omega_{\alpha'',\beta''}}\jap{v}^{\frac{1}{2}}\derv{''}{''}G}_{L^2([0,T];L^2_xL^2_v)}^2+\norm{\jap{v}^{\omega_{\alpha''',\beta'''}}\jap{v}^{\frac{1}{2}}\derv{'''}{'''}G}_{L^2([0,T];L^2_xL^2_v)}^2].
\end{multline*}
\end{lemma}
\begin{proof}
The idea is to estimate $\norm{\h}_{Y^8_v}$ in $L^\infty_x$ and then use Sobolev embedding to go to $L^2_x$ at the cost of two spatial derivatives.
\begin{equation*}
\begin{split}
& \int_0^T \int_{\Omega_R}\int \jap{v}^{\omega_{\alpha'',\beta''}+\omega_{\alpha''',\beta'''}} |\derv{'''}{'''} G|\jap{v}^{\gamma}\norm{\h}_{Y_v^{8}}|\derv{''}{''} G|\d v \d x\d t\\
&\leq \int_0^T \int_{\Omega_R}\int\jap{v}^{\omega_{\alpha'',\beta''}+\omega_{\alpha''',\beta'''}} |\derv{'''}{'''} G|\jap{v}^{\gamma}\norm{\h}_{L^\infty_x Y_v^8}|\derv{''}{''} G|\d v \d x\d t\\
&\leq \norm{\h}_{L^\infty([0,T];L^2_xY_v)}\int_0^T\int_{\Omega_R}\int \jap{v}^{\omega_{\alpha'',\beta''}+\omega_{\alpha''',\beta'''}} |\derv{'''}{'''} G||\derv{''}{''} G|\d v \d x\d t\\
&\lesssim \norm{\h}_{\H} \int_0^T\int_{\Omega_R}\int \jap{v}^{\omega_{\alpha'',\beta''}+\omega_{\alpha''',\beta'''}}|\derv{'''}{'''} G|\jap{v}^{\gamma}|\derv{''}{''} G|\d v \d x\d t.
\end{split}
\end{equation*}
Now we use Cauchy-Schwatrz followed by Young's inequality to get the desired result
\begin{multline*}
\norm{\jap{v}^{\omega_{\alpha'',\beta''}+\omega_{\alpha''',\beta'''}}\derv{'''}{'''}G \jap{v}^\gamma \derv{''}{''} G}_{L^1([0,T];L^1_xL^1_v)}\\
\leq \frac{1}{2}\norm{\jap{v}^{\omega_{\alpha''',\beta'''}}\jap{v}^{\frac{\gamma}{2}}\derv{'''}{'''}G}_{L^2([0,T];L^2_xL^2_v)}^2+\frac{1}{2}\norm{\jap{v}^{\omega_{\alpha'',\beta''}}\jap{v}^{\frac{\gamma}{2}}\derv{''}{''}G}_{L^2([0,T];L^2_xL^2_v)}^2.
\end{multline*}
\end{proof}
\begin{lemma}\label{l.special_holder_der_geq_9}
Let $G$ be a solution to \eref{lin_eq_visc_landau} and $\h=e^{d(t)\jap{v}}h$ where $h$ is the nonnegative function from \lref{lin_visc_landau_exis}. For $|\alpha''|+|\beta''|\leq 8$, we have the following estimate,
\begin{multline*}
\norm{\jap{v}^{\omega_{\alpha,\beta}+\omega_{\alpha'',\beta''}-(3+2\delta_\gamma)}|\der G|\jap{v}^{\gamma} \norm{\h}_{Y^{10}_v}|\derv{''}{''} G|}_{L^1([0,T];L^2_xL^2_v)}\\ \lesssim \norm{\h}_{\H}[\norm{\jap{v}^{\omega_{\alpha,\beta}}\jap{v}^{\frac{1}{2}}\der G}_{L^2([0,T];L^2_xL^2_v)}^2
+ \norm{\jap{v}^{\omega_{\alpha''',\beta''}}\jap{v}^{\frac{1}{2}}\derv{'''}{''}G}_{L^2([0,T];L^2_xL^2_v)}^2],
\end{multline*}
where we sum over $\alpha'''$ such that $|\alpha|\leq |\alpha'''|\leq |\alpha''|+2$.
\end{lemma}
\begin{proof}
\begin{equation*}
\begin{split}
& \int_0^T \int_{\Omega_R}\int\jap{v}^{\omega_{\alpha,\beta}+\omega_{\alpha'',\beta''}-(3+2\delta_\gamma)} |\der G|\jap{v}^{\gamma}\norm{\h}_{Y_v^{10}}|\derv{''}{''} G|\d v \d x\d t\\
&\leq \int_0^T\int_{\Omega_R} (\norm{\h}_{Y_v^{10}}\norm{\jap{v}^{\omega_{\alpha,\beta}}\jap{v}^{\frac{\gamma}{2}} \der G}_{L^2_v}\norm{\jap{v}^{\omega_{\alpha'',\beta''}-(3+2\delta_\gamma)}\jap{v}^{\frac{\gamma}{2}}\derv{''}{''} G}_{L^2_v})\\
&\leq \norm{\h}_{\H} \norm{\jap{v}^{\omega_{\alpha,\beta}}\jap{v}^{\frac{\gamma}{2}} \der G}_{L^2([0,T];L^2_xL^2_v)}\norm{\jap{v}^{\omega_{\alpha'',\beta''}-(3+2\delta_\gamma)}\jap{v}^{\frac{\gamma}{2}} \derv{''}{''} G}_{L^2([0,T];L^\infty_xL^2_v)}\\
&\lesssim \norm{\h}_{\H} \norm{\jap{v}^{\omega_{\alpha,\beta}}\jap{v}^{\frac{\gamma}{2}} \der G}_{L^2([0,T];L^2_xL^2_v)}\norm{\jap{v}^{\omega_{\alpha'',\beta''}-(3+2\delta_\gamma)}\jap{v}^{\frac{\gamma}{2}}\derv{'''}{''} G}_{L^2([0,T];L^2_xL^2_v)}.
\end{split}
\end{equation*}
Here $|\alpha'''|\leq |\alpha''|+2$. Thus $\omega_{\alpha''',\beta''}\geq \omega_{\alpha'',\beta''}-(3+2\delta_\gamma)$ which implies the following bound,
\begin{multline*}
\norm{\h}_{\H} \norm{\jap{v}^{\omega_{\alpha,\beta}}\jap{v}^{\frac{\gamma}{2}} \der G}_{L^2([0,T];L^2_xL^2_v)}\norm{\jap{v}^{\omega_{\alpha'',\beta''}-(3+2\delta_\gamma)}\jap{v}^{\frac{\gamma}{2}}\derv{'''}{''} G}_{L^2([0,T];L^2_xL^2_v)}\\
\leq \norm{\h}_{\H} \norm{\jap{v}^{\omega_{\alpha,\beta}}\jap{v}^{\frac{\gamma}{2}} \der G}_{L^2([0,T];L^2_xL^2_v)}\norm{\jap{v}^{\omega_{\alpha''',\beta''}}\jap{v}^{\frac{\gamma}{2}}\derv{'''}{''} G}_{L^2([0,T];L^2_xL^2_v)}.
\end{multline*}

Using Young's inequality we get,
\begin{multline*}
 \norm{\h}_{\H} \norm{\jap{v}^{\omega_{\alpha,\beta}}\jap{v}^{\frac{\gamma}{2}} \der G}_{L^2([0,T];L^2_vL^2_x)}\norm{\jap{v}^{\omega_{\alpha''',\beta''}}\jap{v}^{\frac{\gamma}{2}}\derv{'''}{''} G}_{L^2([0,T];L^2_xL^2_v)},\\
\leq \frac{1}{2}\norm{\h}_{\H} \norm{\jap{v}^{\omega_{\alpha,\beta}}\jap{v}^{\frac{\gamma}{2}} \der G}_{L^2([0,T];L^2_xL^2_v)}^2+\frac{1}{2}\norm{\h}_{\H}\norm{\jap{v}^{\omega_{\alpha''',\beta''}}\jap{v}^{\frac{\gamma}{2}}\derv{'''}{''} G}_{L^2([0,T];L^2_xL^2_v)}^2.
\end{multline*}
\end{proof}
\section{Bounding Error Terms}\label{s.errors}
We are now in a position to bound all the error terms.\\
Before we start estimating, we make a few notational definitions,
 \begin{equation}\label{e.ball_norm}
\norm{G}_{Y^{m,s}_{l,\Omega_R}}^2:=\sum_{i=0}^m\sum\limits_{|\alpha|+|\beta|=i}\norm{\jap{v}^{2\omega_{\alpha,\beta}}\jap{v}^s(\der G)^2 \psi_{i-l}^2}_{L^1_xL^1_v}.
\end{equation}
And
 \begin{equation}\label{e.ball_norm}
\norm{G}^2_{E^{m}_{T,\Omega_R}}:=\norm{G}^2_{L^\infty([0,T];Y^{m,0}_{0,\Omega_R})}+\int_0^T  \norm{G}^2_{Y^{m,1}_{0,\Omega_R}} \d t.
\end{equation}

Here $m$ denotes the number of derivatives we take, $s$ is the extra velocity weights needed and $l$ is the off-set in the hierarchy of cut-off. This is needed to handle the boundary error terms and $l=1$ in that case. In the case of bulk error terms, $l=0$.

We remind ourselves again that boundary error terms are $0$ when $|\alpha|+|\beta|=0$ thus with the convention that $\psi_0=1$, the above definition makes sense for $l=1$ as soon as $m\geq 1$.

\textbf{In this whole section we fix $\alpha,\beta$ such that $|\alpha|+|\beta|=m\leq 10$.}\\
\textbf{For this section, all the integrals are taken over $\Omega_R$ but the explicit dependence is dropped for the sake of brevity.}
\begin{proposition}\label{p.T1_bound}
We have the following bound on $T_1^{\alpha,\beta}$ in \eref{T1} for all $T \in [0,T_0)$,
$$T_1^{\alpha,\beta}\leq C(\gamma,d_0,m)\int_0^T \norm{G}_{Y^{m,1}_{0,\Omega_R}}^2\d t.$$
\end{proposition}
\begin{proof}
We fix a particular term in $T^{\alpha,\beta}_1$, and assume $\alpha',\beta'$ satisfies the required conditions, that is $|\alpha'|\leq |\alpha'|+1$ and $|\beta'|\leq |\beta|-1$.\\
Now using Young's inequality we have,
\begin{equation*}
\begin{split}
\norm{\jap{v}^{2\omega_{\alpha,\beta}}\psi_m^2|\der G||\derv{'}{'} G|}_{L^1([0,T];L^1_xL^1_v)}&=\frac{1}{2}\norm{\jap{v}^{2\omega_{\alpha,\beta}}\psi_m^2\jap{v}(\der G)^2}_{L^1([0,T];L^1_xL^1_v)}\\
&\quad+\frac{1}{2}\norm{\jap{v}^{2\omega_{\alpha,\beta}-2}\psi_m^2\jap{v}(\derv{'}{'} G)^2}_{L^1([0,T];L^1_xL^1_v)}.
\end{split}
\end{equation*}

Now since $|\alpha'|\leq |\alpha|+1$ and $|\beta'|\leq |\beta|-1$, we thus have $$\omega_{\alpha',\beta'}=20-(\frac{3}{2}+\delta_\gamma)|\alpha'|-(\frac{1}{2}+\delta_\gamma)|\beta'|\geq 20-(\frac{3}{2}+\delta_\gamma)|\alpha|-(\frac{1}{2}+\delta_\gamma)|\beta|-1=\omega_{\alpha,\beta}-1.$$

Since we have that $|\alpha'|+|\beta'|(=i)\leq m$, we must have that $\psi_i\geq \psi_m$. Thus, summing over all such $\alpha',\beta'$, we get the required result.
\end{proof}
\begin{proposition}\label{p.T2_bound}
 We bound $T_2^{\alpha,\beta}$ in \eref{T2} for all $T\in [0,T_0)$ as follows,
$$T_2^{\alpha,\beta}\leq C(d_0,\gamma,m)\kappa\int_0^T \norm{G}_{Y^{m,0}_{0,\Omega_R}}^2.$$
\end{proposition}
\begin{proof}
Again, we fix a particular term in $T^{\alpha,\beta}_2$ with $|\beta'|\leq |\beta|-1$.\\
Using Cauchy-Schwartz and then Young's inequality we have that,
\begin{equation*}
\begin{split}
\norm{\jap{v}^{2\omega_{\alpha,\beta}}\psi_m^2|\der G||\derv{}{'}G|}_{L^1([0,T];L^1_xL^1_v)}&\leq \norm{\jap{v}^{\omega_{\alpha,\beta}}\psi_m\der G}_{L^2([0,T];L^2_xL^2_v)}\norm{\jap{v}^{\omega_{\alpha,\beta}}\psi_m\derv{}{'} G}_{L^2([0,T];L^2_xL^2_v)}\\
&\leq \norm{\jap{v}^{\omega_{\alpha,\beta}}\der G \psi_m}_{L^2([0,T];L^2_xL^2_v)}^2\\
&\quad+\norm{\jap{v}^{\omega_{\alpha,\beta}}\derv{}{'} G\psi_m}_{L^2([0,T];L^2_xL^2_v)}^2.
\end{split}
\end{equation*}
Since $|\beta'|\leq |\beta|-1$ we trivially have $\omega_{\alpha,\beta'}\geq \omega_{\alpha,\beta}$.\\
Summing, we get the required result.
\end{proof}
\begin{proposition}\label{p.T3_1_bound}
For $T^{\alpha,\beta}_{3,1}$ in \eref{T3_1} and $T\in [0,T_0)$ we have the bound,
$$T^{\alpha,\beta}_{3,1}\leq C(d_0,\gamma,m)\norm{\h}_{\H}\int_0^T\norm{G}_{Y^{m,1}_{0,\Omega_R}}^2.$$
\end{proposition}
\begin{proof}
We fix a typical term such that $\alpha',\beta',\alpha'',\beta'',\alpha''',\beta'''$ satisfy the required conditions.

In this case we have less than $8$ derivatives hitting $\a[h]$ and thus we estimate it in $L^\infty_x$ and then use Sobolev embedding at the cost of two derivatives in $x$. Thus the idea is to reduce to a point when we can use \lref{special_holder_der_leq_8} to deduce the required lemma.
\begin{enumerate}
\item \emph{Case 1:} $|\beta'|\geq 2$.\\
We bound $\norm{\jap{v}^{2\omega_{\alpha,\beta}}|\derv{'''}{'''} G||\derv{'}{'}\a[h]||\derv{''}{''}G|}_{L^1([0,T];L^1_xL^1_v)}$ using \eref{L_inf_a_2_der}.
\begin{multline*}
\int_0^T \int \int \jap{v}^{2\omega_{\alpha,\beta}} |\derv{'''}{'''} G||\derv{'}{'} \a[h]||\derv{''}{''} G|\d v \d x \d t\\
\leq \int_0^T \int\int \jap{v}^{2\omega_{\alpha,\beta}} |\derv{'''}{'''} G|\jap{v}^{\gamma}\norm{\h}_{Y^8_v}|\derv{''}{''} G|\d v \d x\d t.
\end{multline*}

Since we have $|\alpha'''|+|\alpha''|+|\alpha'|\leq 2|\alpha|$ and $|\beta'''|+|\beta''|+|\beta'|\leq 2|\beta|+2$ we get that,
\begin{equation*}
\begin{aligned}
\omega_{\alpha''',\beta'''}+\omega_{\alpha'',\beta''}&=40-(\frac{3}{2}+\delta_\gamma)(|\alpha'''|+|\alpha''|)-(\frac{1}{2}+\delta_\gamma)(|\beta'''|+|\beta''|)\\
&\geq 40-(3+2\delta_\gamma)|\alpha|-(1+2\delta_\gamma)|\beta|+\frac{(3+2\delta_\gamma)|\alpha'|+(1+2\delta_\gamma)|\beta'|}{2}-1-2\delta_\gamma.
\end{aligned}
\end{equation*}
In this case $|\beta'|\geq 2$ thus, we get that $$\omega_{\alpha''',\beta'''}+\omega_{\alpha'',\beta''}\geq 2\omega_{\alpha,\beta}.$$

\item \emph{Case 2}: $|\beta'|=1$.\\
In this case we have that $|\alpha'|\geq 1$. Proceeding as above and using \eref{L_inf_a_1_der} we get
\begin{multline*}
\int_0^T \int \int \jap{v}^{2\omega_{\alpha,\beta}} |\derv{'''}{'''} G||\derv{'}{'} \a[h]||\derv{''}{''} G|\d v \d x \d t\\ \lesssim  \int_0^T\int\int \jap{v}^{2\omega_{\alpha,\beta}+1}|\derv{'''}{'''} G|\jap{v}^{\gamma}\norm{\h}_{Y^7_v}|\derv{''}{''} G|\d v \d x\d t.
\end{multline*}
Now since $|\alpha'|\geq 1$, we have by above computations, that $$\omega_{\alpha''',\beta'''}+\omega_{\alpha'',\beta''}\geq 2\omega_{\alpha,\beta}+\frac{(3+2\delta_\gamma)|\alpha'|+(1+2\delta_\gamma)|\beta'|}{2}-1-2\delta_\gamma\geq 2\omega_{\alpha,\beta}+1.$$

Again putting things together we get an equation that can be handled by \lref{special_holder_der_leq_8}.

\item \emph{Case 3:} $|\alpha'|\geq 2$.\\
For this case we use \eref{L_inf_a} to get,
\begin{multline*}
\int_0^T \int \int \jap{v}^{2\omega_{\alpha,\beta}} |\derv{'''}{'''} G||\derv{'}{'} \a[h]||\derv{''}{''} G|\d v \d x \d t\\ \lesssim  \int_0^T\int\int \jap{v}^{2\omega_{\alpha,\beta}+2}|\derv{'''}{'''} G|\norm{\h}_{Y^8_v}\jap{v}^{\gamma}|\derv{''}{''} G|\d v \d x\d t.
\end{multline*}
Since $|\alpha'|\geq 2$, we have $\omega_{\alpha''',\beta'''}+\omega_{\alpha'',\beta''}\geq 2\omega_{\alpha,\beta}+2$.\\
Hence we again get the equation of the desired form.\\
\end{enumerate}

Thus in each case we get a bound of the form,
\begin{multline*}
\int_0^T \int \int \jap{v}^{2\omega_{\alpha,\beta}} |\derv{'''}{'''} G||\derv{'}{'} \a[h]||\derv{''}{''} G|\d v \d x \d t\\
\leq \int_0^T \int\int \jap{v}^{\omega_{\alpha''',\beta'''}+\omega_{\alpha'',\beta''}} |\derv{'''}{'''} G|\jap{v}^{\gamma}\norm{\h}_{Y_v^{8}}|\derv{''}{''} G|\d v \d x\d t.
\end{multline*}
Hence we can apply \lref{special_holder_der_leq_8} as summing over the various indices after noting that $|\alpha''|+|\beta''|\leq m$ and also $|\alpha'''|+|\beta'''|=m$, to get the result of the lemma.
\end{proof}
\begin{proposition}\label{p.T3_2_bound}
For $T^{\alpha,\beta}_{3,2}$ as in \eref{T3_2} and $T\in [0,T_0)$ we have the following estimate,
$$T^{\alpha,\beta}_{3,2}\leq C(d_0,\gamma,m) \norm{\h}_{\H} \int_0^T\norm{G}_{Y^{m,1}_{0,\Omega_R}}^2.$$
\end{proposition}
\begin{proof}
Since we have more than $8$ derivatives hitting $\a[h]$, we  can no longer estimate it in $L^\infty_x$ and we have to estimate it necessarily in $L^2_x$. On the other hand the term $\derv{''}{''} G$ has atmost 4 derivatives hitting it and so we estimate it in $L^\infty_x$ and use Sobolev embedding at the cost of two derivatives. Hence this is where \lref{special_holder_der_geq_9} comes handy.

\begin{enumerate}
\item \emph{Case 1:} $|\beta'|\geq 2$\\
Using \eref{L_inf_a_2_der} we have
\begin{multline*}
\int_0^T \int\int \jap{v}^{2\omega_{\alpha,\beta}} |\der G||\derv{'}{'} \a[h]||\part^2_{v_iv_j}\derv{''}{''} G|\d v \d x \d t\\
\leq \int_0^T \int\int\jap{v}^{2\omega_{\alpha,\beta}} |\der G|\jap{v}^{\gamma}\norm{\h}_{Y_v^{10}}|\derv{''}{'''} G|\d v \d x\d t.
\end{multline*}
Here $|\beta'''|=|\beta''|+2$.

To be able to use \lref{special_holder_der_geq_9} we need to prove that $\omega_{\alpha,\beta}\leq \omega_{\alpha'',\beta'''}-3-2\delta_\gamma$.

Now note that since,
\begin{equation*}
\begin{split}
\omega_{\alpha'',\beta'''}&=20-(\frac{3}{2}+\delta_\gamma)|\alpha''|-(\frac{1}{2}+\delta_\gamma)|\beta'''|\\
&\geq 20-(\frac{3}{2}+\delta_\gamma)|\alpha''|-(\frac{1}{2}+\delta_\gamma)|\beta''|-1-2\delta_\gamma.
\end{split}
\end{equation*}
But,
\begin{equation*}
\begin{split}
\omega_{\alpha'',\beta''}&=20-(\frac{3}{2}+\delta_\gamma)|\alpha''|-(\frac{1}{2}+\delta_\gamma)|\beta''|\\
&=20-(\frac{3}{2}+\delta_\gamma)|\alpha|-(\frac{1}{2}+\delta_\gamma)|\beta|+\frac{(3+2\delta_\gamma)|\alpha'|+(1+2\delta_\gamma)|\beta'|}{2}.
\end{split}
\end{equation*}

Putting this together we get $$\omega_{\alpha'',\beta'''}\geq \omega_{\alpha,\beta}+(3+2\delta_\gamma)\frac{|\alpha'|+|\beta'|}{2}-|\beta'|-1-2\delta_\gamma.$$
Now when $|\alpha'|+|\beta'|= k$, we trivially have $|\beta'|\leq k$, which implies for $9 \leq k\leq 10$.
\begin{equation*}
\begin{split}
\omega_{\alpha'',\beta'''}&\geq \omega_{\alpha,\beta}+\frac{(3+2\delta_\gamma) k}{2}-k-1-2\delta_\gamma\\
&=\omega_{\alpha,\beta}+\frac{k}{2}-1+\delta_\gamma(k-2)\\
&\geq \omega_{\alpha,\beta}+3+2\delta_\gamma.
\end{split}
\end{equation*}

\item \emph{Case 2:} $|\beta'|=1$.\\
In this case we necessarily have $|\alpha'|\geq 8$.\\
Proceeding as in Case 1, and using \eref{L_inf_a_1_der} we get the following bound,
\begin{multline*}
\int_0^T \int \int \jap{v}^{2\omega_{\alpha,\beta}} |\der G||\derv{'}{'} \a[h]||\part^2_{v_iv_j}\derv{''}{''} G|\d v \d x \d t\\
\leq \int_0^T \int\int\jap{v}^{2\omega_{\alpha,\beta}+1} |\der G|\jap{v}^{\gamma}\norm{\h}_{Y_v^{10}}|\derv{''}{'''} G|\d v \d x\d t.
\end{multline*}
Where again, $|\beta'''|=|\beta''|+2$.

This time we need to show $\omega_{\alpha,\beta}+4+2\delta_\gamma\leq \omega_{\alpha'',\beta''}$ (to take care of the extra $\jap{v}$ power).\\
As before  $$\omega_{\alpha'',\beta'''}\geq \omega_{\alpha,\beta}+(3+2\delta_\gamma)\frac{|\alpha'|+|\beta'|}{2}-|\beta'|-1-2\delta_\gamma.$$

But this time since $|\beta'|=1$, $|\alpha'|\geq 8$, thus $$(3+2\delta_\gamma)\frac{|\alpha'|+|\beta'|}{2}-|\beta'|-1-2\delta_\gamma>4+2\delta_\gamma.$$
Hence $\omega_{\alpha,\beta}+4+2\delta_\gamma\leq \omega_{\alpha'',\beta'''}.$

\item \emph{Case 3:} $|\beta'|=0$.\\
In this case we necessarily have $|\alpha'|\geq 9$.\\
Using \eref{L_inf_a}, we get,
\begin{multline*}
\int_0^T \int \int \jap{v}^{2\omega_{\alpha,\beta}} |\der G||\derv{'}{'} \a[h]||\part^2_{v_iv_j}\derv{''}{''} G|\d v \d x \d t\\
\leq \int_0^T \int\int\jap{v}^{2\omega_{\alpha,\beta}+2} |\der G|\jap{v}^{\gamma}\norm{\h}_{Y_v^{10}}|\derv{''}{'''} G|\d v \d x\d t.
\end{multline*}
In this case we need to show $\omega_{\alpha,\beta}+5+2\delta_\gamma\leq \omega_{\alpha'',\beta'''}$.

Again, $$\omega_{\alpha'',\beta'''}\geq \omega_{\alpha,\beta}+(3+2\delta_\gamma)\frac{|\alpha'|+|\beta'|}{2}-|\beta'|-1-2\delta_\gamma.$$

Since $|\beta'|=0$, $(3+2\delta_\gamma)\frac{|\alpha'|+|\beta'|}{2}-1-2\delta_\gamma\geq \frac{27}{2}-1+2\delta_\gamma>5+2\delta_\gamma$.

Hence $\omega_{\alpha,\beta}+5+2\delta_\gamma\leq \omega_{\alpha'',\beta'''}$.
\end{enumerate}
In each case we proved a bound of the form ,
\begin{multline*}
\int_0^T \int \int \jap{v}^{2\omega_{\alpha,\beta}} |\der G||\derv{'}{'} \a[h]||\part^2_{v_iv_j}\derv{''}{''} G|\d v \d x \d t\\
\leq \int_0^T \int\int\jap{v}^{\omega_{\alpha,\beta}+\omega_{\alpha'',\beta'''}-(3+2\delta_\gamma)} |\der G|\jap{v}^{\gamma}\norm{\h}_{Y_v^{10}}|\derv{''}{'''} G|\d v \d x\d t.
\end{multline*}
Thus applying \lref{special_holder_der_geq_9} and summing gives us the required estimate.
\end{proof}
\begin{remark}
In case 1 of \pref{T3_2_bound}, since we have two $\part_v$ hitting $\a[h]$, we can actually apply Sobolev embedding on this term itself but this is no longer true when $|\beta'|<2$.
\end{remark}
\begin{proposition}\label{l.T3_3_bound}
For $T^{\alpha,\beta}_{3,3}$ as in \eref{T3_3} and $T\in [0,T_0)$ we have the following estimate,
$$T^{\alpha,\beta}_{3,3}\leq \norm{\h}_{\H}C(d_0,\gamma,m) \int_0^T\norm{G}_{Y^{m,1}_{0,\Omega_R}}^2.$$
\end{proposition}
\begin{proof}
Each term in $T^{\alpha,\beta}_{3,3}$ only has one derivative hitting $\a[h]$, but we have one less $\jap{v}$ weight to worry about. Let $\alpha',\alpha'',\beta'$ and $\beta''$ satisfy the required conditions.

The idea is the same as that of \pref{T3_1_bound}, and we bound $\derv{'}{'} \a[h]$ in $L^\infty_x$ and then use Sobolev embedding in space.\\
We handle the two cases, when $|\alpha'|=1$ and when $|\beta'|=1$, differently.
\begin{enumerate}
\item \emph{Case 1:} $|\beta'|=1$ (or $|\alpha'|=0$).\\
In this case we proceed as in Case 2 of \pref{T3_1_bound} to get the following bound
\begin{multline*}
\int_0^T \int \int \jap{v}^{2\omega_{\alpha,\beta}-1} |\der G||\derv{'}{'} \a[h]||\part_{v_l}\derv{''}{''} G|\d v \d x \d t\\
\leq \int_0^T \int\int\jap{v}^{2\omega_{\alpha,\beta}} |\der G|\jap{v}^{\gamma}\norm{\h}_{Y_v^{10}}|\derv{''}{'''} G|\d v \d x\d t.
\end{multline*}
Here $|\beta'''|=|\beta''|+1$.\\
We would like to use \lref{special_holder_der_leq_8} but we need to make sure that $\omega_{\alpha'',\beta'''}\geq \omega_{\alpha,\beta}$.\\
Indeed,
$$\omega_{\alpha'',\beta'''}=20-(\frac{3}{2}+\delta_\gamma)|\alpha''|-(\frac{1}{2}+\delta_\gamma)|\beta'''|=\omega_{\alpha'',\beta''}-\frac{1}{2}-\delta_\gamma.$$
But, $$\omega_{\alpha'',\beta''}=\omega_{\alpha,\beta}+\frac{(3+2\delta_\gamma)|\alpha'|+(1+2\delta_\gamma)|\beta'|}{2}=\omega_{\alpha,\beta}+\frac{1}{2}+\delta_\gamma.$$

Thus $\omega_{\alpha'',\beta'''}=\omega_{\alpha,\beta}$.

\item \emph{Case 2:} $|\alpha'|=1$( thus $\beta'=0$).\\
Using \eref{pw_bound_a}, we get the following estimate,
\begin{multline*}
\int_0^T \int\int \jap{v}^{2\omega_{\alpha,\beta}-1} |\der G||\derv{'}{'} \a[h]||\part_{v_l}\derv{''}{''} G|\d v \d x \d t\\
\leq \int_0^T \int\int\jap{v}^{2\omega_{\alpha,\beta}+1} |\der G|\jap{v}^{\gamma}\norm{\h}_{Y_v^{10}}|\derv{''}{'''} G|\d v \d x\d t.
\end{multline*}
Again $|\beta'''|=|\beta''|+1$.\\
To use \lref{special_holder_der_leq_8} we need to prove that $\omega_{\alpha'',\beta'''}\geq \omega_{\alpha,\beta}+1$.

Since $|\alpha'|=1$,we have as above, $$\omega_{\alpha'',\beta'''}=\omega_{\alpha'',\beta''}-\frac{1}{2}-\delta_\gamma=\omega_{\alpha,\beta}+\frac{3}{2}+\delta_\gamma-\frac{1}{2}-\delta_\gamma.$$
Thus $\omega_{\alpha'',\beta'''}=\omega_{\alpha,\beta}+1$.\\
\end{enumerate}
In both cases, we showed a bound which allows us to use \lref{special_holder_der_leq_8} to give us the required estimate.
\end{proof}
\begin{proposition}\label{p.T4_bound}
For $T^{\alpha,\beta}_{4}$ as in \eref{T4} and $T\in [0,T_0)$ we have the following estimate
$$T^{\alpha,\beta}_{4}\leq  \norm{\h}_{\H}C(d_0,\gamma,m) \int_0^T\norm{G}_{Y^{m,1}_{0,\Omega_R}}^2.$$
\end{proposition}
\begin{proof}
Using \lref{L_inf_c} we see that the proof is the same as the first case of \pref{T3_1_bound}.
\end{proof}
\begin{proposition}\label{p.T5_1_bound}
For $T^{\alpha,\beta}_{5,1}$ as in \eref{T5_1} and $T\in [0,T_0)$ we have the following estimate,
$$T^{\alpha,\beta}_{5,1}\leq (\norm{\h}_{\H}+\eps)C(d_0,\gamma,m) \int_0^T\norm{G}_{Y^{m,1}_{0,\Omega_R}}^2.$$
\end{proposition}
\begin{proof}
First we bound the term involving $\eps$. From our restrictions on $\alpha''$ and $\beta''$, we trivially have
\begin{multline*}
\int_0^T \int \int \jap{v}^{2\omega_{\alpha,\beta}} |\der G|\left|\derv{'}{'}\left(\eps  \frac{v_i}{\jap{v}}\right)\right| |\derv{''}{''} G|\d v \d x \d t\\ \lesssim  \eps\norm{\der G\psi_m\jap{v}^{\omega_{\alpha,\beta}}\jap{v}^{\frac{1}{2}}}_{L^2([0,T];L^2_xL^2_v)}+\eps\norm{\derv{''}{''} G\psi_m\jap{v}^{\omega_{\alpha'',\beta''}}\jap{v}^{\frac{1}{2}}}_{L^2([0,T];L^2_xL^2_v)}
\end{multline*}
This gives us the required bound after summing over the multi-indices.

We split this into two cases and each case into two further subcases:
\begin{enumerate}
\item \emph{Case 1:} $|\beta'|+|\alpha'|\leq 8$.\\
In this case we will set up to use \lref{special_holder_der_leq_8}.
\begin{itemize}
\item \emph{Subcase 1a.} $|\beta'|\geq 1$.\\
Using \eref{L_inf_a_v_1_der} we get the following bound,
\begin{multline*}
\int_0^T \int \int \jap{v}^{2\omega_{\alpha,\beta}} |\der G|\left|\derv{'}{'}\left(\a[h] \frac{v_i}{\jap{v}}\right)\right| |\derv{''}{''} G|\d v \d x \d t\\ \lesssim  \int_0^T\int\int \jap{v}^{2\omega_{\alpha,\beta}}|\der G|\norm{\h}_{Y^8_v}\jap{v}^{\gamma}|\derv{''}{''} G|\d v \d x\d t.
\end{multline*}

Now we just need to show that $\omega_{\alpha,\beta}\leq \omega_{\alpha'',\beta''}$.

Since $|\beta'|+|\beta''|=|\beta|+1$ and $|\alpha'|+|\alpha''|=|\alpha|$, we have
$$\omega_{\alpha'',\beta''}=\omega_{\alpha,\beta}+\frac{(3+2\delta_\gamma)|\alpha'|+(1+2\delta_\gamma)|\beta'|}{2}-\frac{1}{2}-\delta_\gamma.$$

In this case $|\beta'|\geq 1$, so we have that $$\omega_{\alpha,\beta}\leq \omega_{\alpha'',\beta''}.$$

\item \emph{Subcase 1b.} $|\beta'|=0$ (which implies, $|\alpha'|\geq 1$)\\
Using \eref{L_inf_a_v} we get,
\begin{multline*}
\int_0^T \int \int \jap{v}^{2\omega_{\alpha,\beta}} |\der G|\left|\derv{'}{'}\left(\a[h] \frac{v_i}{\jap{v}}\right)\right| |\derv{''}{''} G|\d v \d x \d t\\ \lesssim  \int_0^T\int\int \jap{v}^{2\omega_{\alpha,\beta}+1}|\der G|\norm{\h}_{Y^8_v}\jap{v}^{\gamma}|\derv{''}{''} G|\d v \d x\d t.
\end{multline*}

This time we need to prove  $\omega_{\alpha,\beta}+1 \leq \omega_{\alpha'',\beta'''}$.

As above, $$\omega_{\alpha'',\beta''}=\omega_{\alpha,\beta}+\frac{(3+2\delta_\gamma)|\alpha'|+(1+2\delta_\gamma)|\beta'|}{2}-\frac{1}{2}-\delta_\gamma.$$

Since $|\alpha'|\geq 1$, we get that $\omega_{\alpha'',\beta''}\geq \omega_{\alpha,\beta}+1.$
\end{itemize}
In both cases we proved an inequality of the form,
\begin{multline*}
\int_0^T \int \int \jap{v}^{2\omega_{\alpha,\beta}} |\der G|\left|\derv{'}{'}\left(\a[h] \frac{v_i}{\jap{v}}\right)\right| |\derv{''}{''} G|\d v \d x \d t\\ \lesssim  \int_0^T\int\int \jap{v}^{\omega_{\alpha,\beta}+\omega_{\alpha'',\beta''}}|\der G|\norm{\h}_{Y^8_v}\jap{v}^{\gamma}|\derv{''}{''} G|\d v \d x\d t.
\end{multline*}
Hence we can use \lref{special_holder_der_leq_8} to prove the required result.\\
\item \emph{Case 2:} $|\alpha'|+|\beta'|\geq 9$
In this case we will set up to use \lref{special_holder_der_geq_9}
\begin{itemize}
\item \emph{Subcase 2a.} $|\beta'|\geq 1$.\\
Again using \eref{L_inf_a_v_1_der} we get the bound,
\begin{multline*}
\int_0^T \int \int \jap{v}^{2\omega_{\alpha,\beta}} |\der G|\left|\derv{'}{'}\left(\a[h] \frac{v_i}{\jap{v}}\right)\right| |\derv{''}{''} G|\d v \d x \d t\\ \lesssim  \int_0^T\int\int \jap{v}^{2\omega_{\alpha,\beta}}|\der G|\norm{\h}_{Y^{10}_v}\jap{v}^{\gamma}|\derv{''}{''} G|\d v \d x\d t.
\end{multline*}

To be able to use \lref{special_holder_der_geq_9} we need to prove that $\omega_{\alpha,\beta}+3+2\delta_\gamma\leq \omega_{\alpha'',\beta''}$.

But note as above that $$\omega_{\alpha'',\beta''}=\omega_{\alpha,\beta}+\frac{3+2\delta_\gamma}{2}(|\alpha'|+|\beta'|)-|\beta'|-\frac{1}{2}-\delta_\gamma.$$ And since $|\alpha'|+|\beta'|\geq 9$ and $|\beta'|\leq |\alpha'|+|\beta'|$, we get that $$\omega_{\alpha'',\beta''}\geq \omega_{\alpha,\beta}+3+2\delta_\gamma.$$

\item \emph{Subcase 2b.} $|\beta'|=0$ (thus $|\alpha'|\geq 9$).\\
Using \eref{L_inf_a_v} we get the following bound,
\begin{multline*}
\int_0^T \int \int \jap{v}^{2\omega_{\alpha,\beta}} |\der G|\left|\derv{'}{'}\left(\a[h] \frac{v_i}{\jap{v}}\right)\right| |\derv{''}{''} G|\d v \d x \d t\\ \lesssim  \int_0^T\int\int \jap{v}^{2\omega_{\alpha,\beta}+1}|\der G|\norm{\h}_{Y^{10}_v}\jap{v}^{\gamma}|\derv{''}{''} G|\d v \d x\d t.
\end{multline*}

Now we need to prove that $\omega_{\alpha,\beta}+4+2\delta_\gamma\leq \omega_{\alpha'',\beta''}$.

Since we have that  $$\omega_{\alpha'',\beta''}=\omega_{\alpha,\beta}\frac{3+2\delta_\gamma}{2}(|\alpha'|+|\beta'|)-|\beta'|-\frac{1}{2}-\delta_\gamma,$$
we get by noting that, $|\beta'|=0$ and $|\alpha'|\geq 9$, $$\omega_{\alpha,\beta}+4+2\delta_\gamma\leq \omega_{\alpha'',\beta''}.$$
\end{itemize}
In both cases we were able to prove a bound that let's us use \lref{special_holder_der_geq_9} and finish the proof.
\end{enumerate}
\end{proof}
\begin{proposition}\label{p.T5_2_bound}
For $T^{\alpha,\beta}_{5,2}$ as in \eref{T5_2} and $T\in [0,T_0)$ we have the following estimate,
$$T^{\alpha,\beta}_{5,2}\leq (\norm{\h}_{\H}+\eps)C(d_0,\gamma,m) \int_0^T\norm{G}_{Y^{m,1}_{0,\Omega_R}}^2.$$
\end{proposition}
\begin{proof}
We again bound the term involving $\eps$ as before.

Although we don't enjoy the reduced number of velocity weights as in \eref{L_inf_a_v_1_der}, we have one less velocity to worry about from the beginning.\\
Using \eref{pw_bound_a_v}, we get the bound,
\begin{multline*}
\int_0^T \int \int \jap{v}^{2\omega_{\alpha,\beta}-1} |\der G|\left|\left(\a[h] \frac{v_i}{\jap{v}}\right)\right| |\der G|\d v \d x \d t\\ \lesssim  \int_0^T\int\int \jap{v}^{2\omega_{\alpha,\beta}}|\der G|\norm{\h}_{Y^{0}_v}\jap{v}^{\gamma}|\der G|\d v \d x\d t.
\end{multline*}
Now the required lemma follows directly from \lref{special_holder_der_leq_8}.
\end{proof}
\begin{proposition}\label{p.T6_1_bound}
For $T^{\alpha,\beta}_{6,1}$ as in \eref{T6_1} and $T\in [0,T_0)$ we have the following estimate,
$$T^{\alpha,\beta}_{6,1}\leq (\norm{\h}_{\H}+\eps)C(d_0,\gamma,m) \int_0^T\norm{G}_{Y^{m,1}_{0,\Omega_R}}^2.$$
\end{proposition}
\begin{proof}
We get the required bound for the $\eps$ term trivially. Indeed,
\begin{multline*}
\int_0^T \int \int \jap{v}^{2\omega_{\alpha,\beta}} |\der G|\left|\derv{'}{'}\left(\frac{\eps}{\jap{v}}\right)\right| |\derv{''}{''} G|\d v \d x \d t\\ \lesssim  \eps\norm{\der G\psi_m\jap{v}^{\omega_{\alpha,\beta}}\jap{v}^{\frac{1}{2}}}_{L^2([0,T];L^2_xL^2_v)}+\eps\norm{\derv{''}{''} G\psi_m\jap{v}^{\omega_{\alpha'',\beta''}}\jap{v}^{\frac{1}{2}}}_{L^2([0,T];L^2_xL^2_v)}
\end{multline*}
Following our ususal line of reasoning, we split this lemma into two cases and further subdivide into two subcases.
\begin{enumerate}
\item \emph{Case 1:} $|\alpha'|+|\beta'|\leq 8$.
\begin{itemize}
\item \emph{Subcase 1a.} $|\beta|\geq 1$.\\
Using \eref{L_inf_a_1_der} with the slight change coming from the already existing $\frac{1}{\jap{v}}$ weight, we get the following bound,
\begin{multline*}
\int_0^T \int \int \jap{v}^{2\omega_{\alpha,\beta}} |\der G|\left|\derv{'}{'}\left(\frac{\bar{a}_{ii}[h]}{\jap{v}}\right)\right| |\derv{''}{''} G|\d v \d x \d t\\ \lesssim  \int_0^T\int\int \jap{v}^{2\omega_{\alpha,\beta}}|\der G|\norm{\h}_{Y^{8}_v}\jap{v}^{\gamma}|\derv{''}{''} G|\d v \d x\d t.
\end{multline*}
Since $|\alpha''|\leq |\alpha|$ and $|\beta''|\leq|\beta|$, we trivially have the requirements for \lref{special_holder_der_leq_8} satisfied.

\item \emph{Subcase 2a.} $|\beta'|=0$ (which means $|\alpha'|\geq 1$).\\
Using \eref{L_inf_a}, we get the following bound,
\begin{multline*}
\int_0^T \int \int \jap{v}^{2\omega_{\alpha,\beta}} |\der G|\left|\derv{'}{'}\left(\frac{\bar{a}_{ii}[h]}{\jap{v}}\right)\right| |\derv{''}{''} G|\d v \d x \d t\\ \lesssim  \int_0^T\int\int \jap{v}^{2\omega_{\alpha,\beta}+1}|\der G|\norm{\h}_{Y^{8}_v}\jap{v}^{\gamma}|\derv{''}{''} G|\d v \d x\d t.
\end{multline*}
Again, since $|\alpha''|\leq |\alpha|$, $|\beta''|\leq|\beta|$ and $|\alpha'|\geq 1$, it is easily seen that $\omega_{\alpha,\beta}+1\leq \omega_{\alpha'',\beta''}$. Indeed, 
$$\omega_{\alpha'',\beta''}=\omega_{\alpha,\beta}+(\frac{3}{2}+\delta_\gamma)|\alpha'|+(\frac{1}{2}+\delta_\gamma)|\beta'|.$$
\end{itemize}
In both cases we are in good shape to apply \lref{special_holder_der_leq_8} whose application gives the required result.

\item \emph{Case 2:} $|\alpha'|+|\beta'|\geq 9$.\\
\begin{itemize}
\item \emph{Subcase 2a.} $|\beta'|\geq 1$.\\
Using \eref{L_inf_a_1_der} with appropriate changes we get the bound,
\begin{multline*}
\int_0^T \int \int \jap{v}^{2\omega_{\alpha,\beta}} |\der G|\left|\derv{'}{'}\left(\frac{\bar{a}_{ii}[h]}{\jap{v}}\right)\right| |\derv{''}{''} G|\d v \d x \d t\\ \lesssim  \int_0^T\int\int \jap{v}^{2\omega_{\alpha,\beta}}|\der G|\norm{\h}_{Y^{10}_v}\jap{v}^{\gamma}|\derv{''}{''} G|\d v \d x\d t.
\end{multline*}
Again we have, $$\omega_{\alpha'',\beta''}=\omega_{\alpha,\beta}+(\frac{3}{2}+\delta_\gamma)(|\alpha'|+|\beta'|)-|\beta'|.$$

Since $|\alpha'|+|\beta'|\geq 9$ and $|\beta'|\leq 10$ we get that $$\omega_{\alpha,\beta}+3+2\delta_\gamma \leq \omega_{\alpha'',\beta''}.$$

\item \emph{Subcase 2b.} $|\beta'|=0$ (or $|\alpha'|\geq 9$).\\
Using \eref{pw_bound_a} we get,
\begin{multline*}
\int_0^T \int \int \jap{v}^{2\omega_{\alpha,\beta}} |\der G|\left|\derv{'}{'}\left(\frac{\bar{a}_{ii}[h]}{\jap{v}}\right)\right| |\derv{''}{''} G|\d v \d x \d t\\ \lesssim  \int_0^T\int\int \jap{v}^{2\omega_{\alpha,\beta}+1}|\der G|\norm{\h}_{Y^{10}_v}\jap{v}^{\gamma}|\derv{''}{''} G|\d v \d x\d t.
\end{multline*}
Since $|\alpha'|\geq 9$ and $|\beta'|=0$, we get easily that $$\omega_{\alpha,\beta}+4+2\delta_\gamma\leq \omega_{\alpha'',\beta''}.$$
\end{itemize}
Finally applying \lref{special_holder_der_geq_9} in both cases gives us the required lemma.
\end{enumerate}
\end{proof}
\begin{proposition}\label{p.T6_2_bound}
For $T^{\alpha,\beta}_{6,2}$ as in \eref{T6_2} and $T\in [0,T_0)$ we have the following estimate,
$$T^{\alpha,\beta}_{6,2}\leq (\norm{\h}_{\H}+\eps)C(d_0,\gamma,m) \int_0^T\norm{G}_{Y^{m,1}_{0,\Omega_R}}^2.$$
\end{proposition}
\begin{proof}
As before, the term with $\eps$ poses no problem and is bounded in a similar way as in \pref{T6_1_bound}.

Now we look at the term involving the coefficient matrix.
\begin{enumerate}
\item \emph{Case 1:} $|\alpha'|+|\beta'|\leq 8$.\\
Using \eref{L_inf_a_2_v} we get the following bound,
\begin{multline*}
\int_0^T \int \int \jap{v}^{2\omega_{\alpha,\beta}} |\der G|\left|\derv{'}{'}\left( \frac{\a[h]v_iv_j}{\jap{v}^2}\right)\right| |\derv{''}{''} G|\d v \d x \d t\\ \lesssim  \int_0^T\int\int \jap{v}^{2\omega_{\alpha,\beta}}|\der G|\norm{\h}_{Y^8_v}\jap{v}^{\gamma}|\derv{''}{''} G|\d v \d x\d t.
\end{multline*}
Trivially, we have $\omega_{\alpha,\beta}\leq \omega_{\alpha'',\beta''}$. Thus applying \lref{special_holder_der_leq_8} gives us the desired inequality.

\item \emph{Case 2:} $|\alpha'|+|\beta'|\geq 9$.\\
Using \eref{L_inf_a_2_v} we get the following bound,
\begin{multline*}
\int_0^T \int \int \jap{v}^{2\omega_{\alpha,\beta}} |\der G|\left|\derv{'}{'}\left( \frac{\a[h]v_iv_j}{\jap{v}^2}\right)\right| |\derv{''}{''} G|\d v \d x \d t\\ \lesssim  \int_0^T\int\int \jap{v}^{2\omega_{\alpha,\beta}}|\der G|\norm{\h}_{Y^{10}_v}\jap{v}^{\gamma}|\derv{''}{''} G|\d v \d x\d t.
\end{multline*}
Since $\omega_{\alpha'',\beta''}=\omega_{\alpha,\beta}+(\frac{3}{2}+\delta_\gamma)(|\alpha'|+|\beta'|)-|\beta'|$, $|\alpha'|+|\beta'|\geq 9$ and $|\beta'|\leq 10$ we have that $$\omega_{\alpha,\beta}+3+2\delta_\gamma\leq \omega_{\alpha'',\beta''}.$$
Thus applying \lref{special_holder_der_geq_9} gives us the required result.
\end{enumerate}
\end{proof}
Now we estimate the final bulk error term that comes up when we apply integration by parts at the hghest order.
\begin{proposition}\label{p.A2}
For $\mathcal{A}_2^{\alpha,\beta}$ as in \eref{a2}, and $T \in [0,T_0)$, we have the following estimate,
$$\mathcal{A}_2^{\alpha,\beta} \leq \norm{\h}_{\H}C(d_0,\gamma,m) \int_0^T\norm{G}_{Y^{m,1}_{0,\Omega_R}}^2.$$
\end{proposition}
\begin{proof}
Although, we have no derivative hitting $\a[h]$ but we have two less $\jap{v}$ weights to handle. Thus using \eref{L_inf_a} we have,
$$\norm{\jap{v}^{2\omega_{\alpha,\beta}-2}\psi_m^2\a[h](\der G)^2}_{L^1([0,T];L^1_vL^1_x)} \leq \norm{\jap{v}^{2\omega_{\alpha,\beta}}\psi_m^2\norm{\h}_{Y^0_v}\jap{v}^{\gamma}(\der G)^2}_{L^1([0,T];L^1_vL^1_x)}.$$
Now using \lref{special_holder_der_leq_8} we get the required result.
\end{proof}
We now begin bounding the boundary error terms. Note that the term with $\eps$, has no dependence on $\jap{v}$ and so we can get decay of that term in $R$ by a trivial adaptation of the proofs below. Thus we safely ignore that term in our propositions below
\begin{proposition}\label{p.B_1}
For $\mathcal{B}_1^{\alpha,\beta}$ in \eref{b1} and $T\in [0,T_0)$ we have the estimate,
$$\mathcal{B}_1^{\alpha,\beta}\lesssim R^{\gamma-1-\delta_\gamma}\norm{\h}_{\H}[\int_0^T\norm{G}_{Y^{m,1}_{0,\Omega_R}}^2+\int_0^T\norm{G}_{Y^{m,1+2\delta_\gamma}_{1,\Omega_R}}^2].$$
\end{proposition}
\begin{proof}
Using \eref{L_inf_a}, the fact that $\part^2_{v_iv_j}\psi_m\lesssim R^{-2}\psi_{m-1}$ and that $\jap{v}\leq 2R$ we get the bound,
\begin{multline*}
\int_0^T \int \int \jap{v}^{2\omega_{\alpha,\beta}}\psi_m \part^2_{v_iv_j}\psi_m(\a[h])(\der G)^2 \d v \d x\d t\\
\lesssim R^{\gamma-1-\delta_\gamma}\int_0^T \int \int \jap{v}^{2\omega_{\alpha,\beta}}\jap{v}^{1+\delta_\gamma}\norm{\h}_{Y^0_v}\psi_{m}\psi_{m-1}(\der G)^2 \d v \d x \d t.
\end{multline*}
Applying Sobolev embedding and Cauchy Schwatrz we get the following inequality,
\begin{multline*}
R^{\gamma-1-\delta_\gamma}\int_0^T \int \int \jap{v}^{2\omega_{\alpha,\beta}}\jap{v}^{1+\delta_\gamma}\norm{\h}_{H^0_v}\psi_{m}\psi_{m-1}(\der G)^2 \d v \d x \d t,\\
\lesssim R^{\gamma-1-\delta_\gamma}\norm{\h}_{\E}[\norm{\jap{v}^{\omega_{\alpha,\beta}}\jap{v}^{\frac{1}{2}}\psi_m \der G}_{L^2([0,T];L^2_vL^2_x)}^2+\norm{\jap{v}^{\omega_{\alpha,\beta}}\jap{v}^{\frac{1}{2}+\delta_\gamma}\psi_{m-1} \der G}_{L^2([0,T];L^2_vL^2_x)}^2].
\end{multline*}
\end{proof}

\begin{proposition}\label{p.B_2}
For $\mathcal{B}_2^{\alpha,\beta}$ in \eref{b2} and $T\in [0,T_0)$ we have the estimate,
$$\mathcal{B}_2^{\alpha,\beta}\lesssim R^{\gamma-1-\delta_\gamma}\norm{\h}_{\H}\int_0^T\norm{G}_{Y^{m,1+2\delta_\gamma}_{1,\Omega_R}}^2.$$
\end{proposition}
\begin{proof}

Using \eref{L_inf_a}, $\part_{v_i}\psi_m\lesssim R^{-1}\psi_{m-1}$ and that $\jap{v}\leq 2R$ we get,
\begin{align*}
\int_0^T \int \int \jap{v}^{2\omega_{\alpha,\beta}}&\part_{v_k}\psi_m \part_{v_l}\psi_m(\a[h])(\der G)^2 \d v \d x\d t\\
&\lesssim R^{\gamma-1-\delta_\gamma}\int_0^T \int \int \jap{v}^{2\omega_{\alpha,\beta}}\jap{v}^{1+\delta_\gamma}\norm{\h}_{Y^0_v}\psi_{m-1}^2(\der G)^2 \d v \d x \d t\\
&\lesssim R^{\gamma-1-\delta_\gamma}\norm{\h}_{\E}\norm{\jap{v}^{\omega_{\alpha,\beta}}\jap{v}^{\frac{1}{2}+\delta_\gamma}\der G}_{L^2([0,T];L^2_vL^2_x)}^2.
\end{align*}
\end{proof}
\begin{proposition}\label{p.B_3}
For $\mathcal{B}_3^{\alpha,\beta}$ in \eref{b3} and $T\in [0,T_0)$ we have the estimate,
$$\mathcal{B}_3^{\alpha,\beta}\lesssim R^{\gamma-1-\delta_\gamma}\norm{\h}_{\H}[\int_0^T\norm{G}_{Y^{m,1}_{0,\Omega_R}}^2+\int_0^T\norm{G}_{Y^{m,1+2\delta_\gamma}_{1,\Omega_R}}^2].$$
\end{proposition}
\begin{proof}
Proceeding as in \pref{B_1} we get,
\begin{multline*}
\int_0^T \int \int \jap{v}^{2\omega_{\alpha,\beta}-1}\psi_m \part_{v_l}\psi_m(\a[h])(\der G)^2 \d v \d x\d t\\
\lesssim R^{\gamma-1-\delta_\gamma}\int_0^T \int \int \jap{v}^{2\omega_{\alpha,\beta}}\jap{v}^{1+\delta_\gamma}\norm{\h}_{Y^0_v}\psi_m\psi_{m-1}(\der G)^2 \d v \d x \d t.
\end{multline*}
Which again as in \pref{B_1} implies the required bound.
\end{proof}
\begin{proposition}\label{p.B_4}
For $\mathcal{B}_4^{\alpha,\beta}$ in \eref{b4} and $T\in [0,T_0)$ we have the estimate,
$$\mathcal{B}_4^{\alpha,\beta}\lesssim R^{\gamma-1-\delta_\gamma}\norm{\h}_{\H}[\int_0^T\norm{G}_{Y^{m,1}_{0,\Omega_R}}^2+\int_0^T\norm{G}_{Y^{m,1+2\delta_\gamma}_{1,\Omega_R}}^2].$$
\end{proposition}
\begin{proof}
Working with a typical term and assuming that $\alpha',\beta',\alpha'',\beta'',\alpha''',\beta'''$ satisy the required conditions, we break the proof into two cases,
\begin{enumerate}
\item \emph{Case 1:} $|\beta'|=1$.\\
In this case we have that $\omega_{\alpha'',\beta''}+\omega_{\alpha''',\beta'''}=2\omega_{\alpha,\beta}$.

Using this, \eref{L_inf_a_1_der} and that $\part_{v_l}\psi_m\lesssim R^{-1}\psi_{m-1}$ we get,
\begin{multline*}
\int_0^T \int \int \jap{v}^{2\omega_{\alpha,\beta}}\psi_m \part_{v_l}\psi_m\derv{'''}{'''} G(\derv{'}{'}\a[h])\derv{''}{''} G\d v \d x\d t\\
\lesssim R^{\gamma-1-\delta_\gamma}\int_0^T \int \int \jap{v}^{\omega_{\alpha'',\beta''}+\omega_{\alpha''',\beta'''}}\jap{v}^{1+\delta_\gamma}\norm{\h}_{Y^0_v}\psi_m\psi_{m-1}\derv{'''}{'''} G\derv{''}{''} G \d v \d x \d t.
\end{multline*}
Using Sobolev embedding and Cauchy Schwartz we get,
\begin{multline*}
\int_0^T \int \int \jap{v}^{2\omega_{\alpha,\beta}}\psi_m \part_{v_l}\psi_m\derv{'''}{'''} G(\derv{'}{'}\a[h])\derv{''}{''} G \d v \d x \d t\\
\lesssim R^{\gamma-1-\delta_\gamma}\norm{\h}_{\H}[\norm{\jap{v}^{\omega_{\alpha''',\beta'''}}\jap{v}^{\frac{1}{2}}\psi_m \derv{'''}{'''} G}_{L^2([0,T];L^2_vL^2_x)}^2+\norm{\jap{v}^{\omega_{\alpha'',\beta''}}\jap{v}^{\frac{1}{2}+\delta_\gamma}\psi_{m-1} \derv{''}{''} G}_{L^2([0,T];L^2_vL^2_x)}^2].
\end{multline*}

\item \emph{Case 2:} $|\alpha'|=1$.\\
Using \eref{L_inf_a} and $\part_{v_l}\psi_m\lesssim R^{-1}\psi_{m-1}$ we get,
\begin{multline*}
\int_0^T \int \int \jap{v}^{2\omega_{\alpha,\beta}}\psi_m \part_{v_l}\psi_m\derv{'''}{'''} G(\derv{'}{'}\a[h])\derv{''}{''} G\d v \d x\d t\\
\lesssim R^{\gamma-1-\delta_\gamma}\int_0^T \int \int \jap{v}^{2\omega_{\alpha,\beta}+1}\jap{v}^{1+\delta_\gamma}\norm{\h}_{Y^0_v}\psi_m\psi_{m-1}\derv{'''}{'''} G\derv{''}{''} G \d v \d x \d t.
\end{multline*}
But since $|\alpha'|=1$,we have in this case $\omega_{\alpha'',\beta''}+\omega_{\alpha''',\beta'''}=2\omega_{\alpha,\beta}+1$.\\
Now proceeding in the same way as in the case above we get the required result.
\end{enumerate}
\end{proof}
\begin{proposition}\label{p.B_5}
For $\mathcal{B}_5^{\alpha,\beta}$ in \eref{b5} and $T\in [0,T_0)$ we have the estimate,
$$\mathcal{B}_5^{\alpha,\beta}\lesssim R^{\gamma-1-\delta_\gamma}\norm{\h}_{\H}[\int_0^T\norm{G}_{Y^{m,1}_{0,\Omega_R}}^2+\int_0^T\norm{G}_{Y^{m,1+2\delta_\gamma}_{1,\Omega_R}}^2].$$
\end{proposition}
\begin{proof}
Using \eref{pw_bound_a_v} that $\part_{v_l}\psi_m\lesssim R^{-1}\psi_{m-1}$ we get,
\begin{multline*}
d(t)\int_0^T \int \int \jap{v}^{2\omega_{\alpha,\beta}}\psi_m \part_{v_l}\psi_m\left(\a[h] \frac{v_i}{\jap{v}}\right)(\der G)^2\d v \d x\d t\\
\lesssim  R^{\gamma-1-\delta_\gamma}\int_0^T \int \int \jap{v}^{2\omega_{\alpha,\beta}}\jap{v}^{1+\delta_\gamma}\norm{\h}_{Y^0_v}\psi_m\psi_{m-1}(\der G)^2 \d v \d x \d t.
\end{multline*}
Now proceeding as in \pref{B_1} gives us the desired result.
\end{proof}
\section{Putting everything together}\label{s.put}
By choosing $\eps$ small enough and using the estimates from last section and \eref{main_energy_estimate} we get,
\begin{equation*}
\begin{split}
&\norm{\jap{v}^{\omega_{\alpha,\beta}}\der G \psi_m}_{L^2_vL^2_x}^2(T)+\kappa \norm{\jap{v}^{\omega_{\alpha,\beta}}\jap{v}^{\frac{1}{2}}\der G \psi_m}_{L^2([0,T];L^2_vL^2_x)}^2\\
&\quad+\eps\norm{\jap{v}^{\omega_{\alpha,\beta}}\part_{x_i}\der G \psi_m}_{L^2([0,T];L^2_vL^2_x)}^2+\eps\norm{\jap{v}^{\omega_{\alpha,\beta}}\part_{v_i}\der G \psi_m}_{L^2([0,T];L^2_vL^2_x)}^2\\
&\leq \norm{\jap{v}^{\omega_{\alpha,\beta}}\der g_{\ini}}_{L^2_vL^2_x}^2+C(d_0,\gamma,m)[\norm{\h}_{\H}+R^{\gamma-1-\delta_\gamma}]\int_0^T\norm{G}_{Y^{m,1}_{0,\Omega_R}}^2\\
&\quad+C(d_0,\gamma,m)\kappa\int_0^T\norm{G}_{Y^{m,0}_{0,\Omega_R}}^2+C(d_0,\gamma,m)R^{\gamma-1-\delta_\gamma}\norm{\h}_{\H}\int_0^T\norm{G}_{Y^{m,1+2\delta_\gamma}_{1,\Omega_R}}^2.
\end{split}
\end{equation*}
\begin{remark}
For $m=0$ the boundary terms are zero since $\psi_0=1$ on $\Omega_R$.

Moreover, the term coming from the viscosity has an extra derivative, so even when the derivative is $\part_v$ we have that $\omega_{\alpha,\beta'}=\omega_{\alpha,\beta}-\frac{1}{2}-\delta_\gamma$. Thus we bound this term with atleast an extra $\jap{v}^{1+2\delta_\gamma}$ weight in $L^2([0,T];L^2_vL^2_x)$.

\end{remark}
Summing over all $\alpha,\beta$ such that $|\alpha|+|\beta|\leq m$ and noting that $\psi_m\leq \psi_i$ for $i\leq m$, we get,
\begin{equation*}
\begin{split}
&\norm{G}_{Y^{m,0}_{0,\Omega_R}}^2(T)+\kappa\int_0^T \norm{G}_{Y^{m,1}_{0,\Omega_R}}^2(t)\d t+\eps \int_0^T \norm{G}_{Y^{m+1,1+2\delta_\gamma}_{1,\Omega_R}}^2(t)\d t\\
&\leq \norm{g_{\ini}}_{Y_{x,v}}^2+C(d_0,\gamma,m)[\norm{\h}_{\H}+R^{\gamma-1-\delta_\gamma}]\int_0^T\norm{G}_{Y^{m,1}_{0,\Omega_R}}^2\d t\\
&\quad+C(d_0,\gamma)\kappa\int_0^T\norm{G}_{Y^{m,0}_{0,\Omega_R}}^2\d t+C(d_0,\gamma,m)R^{\gamma-1-\delta_\gamma}\norm{\h}_{\H}\int_0^T\norm{G}_{Y^{m,1+2\delta_\gamma}_{1,\Omega_R}}^2 \d t.
\end{split}
\end{equation*}

We can absorb the second term on the left hand side by $\kappa \norm{\jap{v}^{\omega_{\alpha,\beta}}\jap{v}^{\frac{1}{2}}\der G \psi_i}_{L^2([0,T];L^2_vL^2_x)}^2$ by choosing $R$ large enough, and noting that the constant is independent of $\kappa$. That is we can apriori make $\kappa$ large enough to absorb this term, although this choice must depend on the $M_h$ from \lref{lin_eq}, since $\norm{\h}_{\H}$ does. More concretely, let $\underline{\kappa=C(d_0,\gamma)M_h}$.\\
\textbf{Since the dependence of $\kappa$ on $M_h$ is linear, we can safely substitute the dependence of constants on $\kappa$ by $M_h$}.
\begin{equation}
\begin{split}
&\norm{G}_{Y^{m,0}_{0,\Omega_R}}^2(T)+\int_0^T \norm{G}_{Y^{m,1}_{0,\Omega_R}}^2\d t+\eps \int_0^T \norm{G}^2_{Y^{m+1,1+2\delta_\gamma}_{1,\Omega_R}}(t)\d t\\
& \leq \norm{g_{\ini}}_{Y_{x,v}}^2+C(d_0,\gamma)M_h\int_0^T \norm{G}^2_{Y^{m,0}_{0,\Omega_R}}(t) \d t+C(d_0,\gamma,m)R^{\gamma-1-\delta_\gamma}\norm{\h}_{\H}\int_0^T \norm{G}^2_{Y^{m,1+2\delta_\gamma}_{1,\Omega_R}}(t)\d t.
\end{split}
\end{equation}

For $m=0$, the term $R^{\gamma-1-\delta_\gamma}\int_0^T \norm{G}^2_{Y^{m,1+2\delta_\gamma}_{1,\Omega_R}}(t)\d t=0$ and so our goal now is to prove that for $R$ large enough we have the estimate,
\begin{equation}\label{e.small} 
C(d_0,\gamma,m)R^{\gamma-1-\delta_\gamma}\norm{\h}_{\H}\int_0^T \norm{G}^2_{Y^{m,1+2\delta_\gamma}_{1,\Omega_R}}(t)\d t\leq \eps.
\end{equation}
 for all $m\leq 10$.

By induction, assume \eref{small} is true for $m-1$, then we have,
\begin{equation*}
\begin{split}
&\norm{G}^2_{Y^{m-1,0}_{0,\Omega_R}}(T)+\int_0^T \norm{G}_{Y^{m,1}_{0,\Omega_R}}^2\d t+\eps \int_0^T \norm{G}^2_{Y^{m,1+2\delta_\gamma}_{1,\Omega_R}}(t)\d t\\
& \leq \norm{g_{\ini}}^2_{\E}+C(d_0,\gamma)M_h\int_0^T \norm{G}^2_{Y^{m-1,0}_{0,\Omega_R}}(t) \d t+\eps.
\end{split}
\end{equation*}
Using Gr\"{o}nwall's inequality we get,
$$\norm{G}^2_{Y^{m-1,0}_{0,\Omega_R}}(T)+\int_0^T \norm{G}_{Y^{m,1}_{0,\Omega_R}}^2\d t+\eps \int_0^T \norm{G}^2_{Y^{m,1+2\delta_\gamma}_{1,\Omega_R}}(t)\d t\leq (\norm{g_{\ini}}^2_{Y_{x,v}}+\eps)\exp(C(d,\gamma)M_hT).$$

The bound on the second term implies \eref{small} for all $m\leq 10$ when $$R\geq 2\left(\frac{1}{\eps^2}(\norm{g_{\ini}}^2_{Y_{x,v}}+\eps)\exp(C(d_0,\gamma)M_hT)\right)^\frac{1}{1+\delta_\gamma-\gamma}$$
\begin{proof}[Proof of \lref{lin_eq}]
Takin $\sup$ over $t\in [0,T]$ we get,
\begin{equation}\label{e.gron}
\norm{G_R}_{E_{T,\Omega_R}}^2\leq (\norm{g_{\ini}}^2_{Y_{x,v}}+\eps)\exp(C(d_0,\gamma)M_hT ).
\end{equation}
Now consider the sequence of functions $\{G_K\}$ for $K\in \N$ of solutions to \eref{lin_eq_visc_landau} on $\Omega_K$ and with $\eps=(\ln K)^{-1}$. Note that this choice of $\eps$ still allows condition \eref{small} to hold, for sufficiently large $K$.

The bound \eref{gron} holds for all such $G_K$. Fix an $L>0$ and remember that $\chi_L$ be a smooth cut-off function on $\R^6$, supported on $\Omega_{L-1}$, equal to $1$ in $\Omega_{L-2}$, radially symmetric, monotone and such that for $|\alpha|+|\beta|=n$, $|\der \chi_R|<2^n$.

For large enough $K$ we have that $\norm{\chi_L G_K}_{\E}$ is uniformly bounded in $K$ (we need $K$ to be large enough so that $\psi_{10}\geq \chi_L$ which implies $\norm{\chi_L G_K}_{\E}\leq \norm{G_K}_{E_{T,\Omega_K}}$). From the trivial bound , $\norm{\chi_L G_K}_{X_T}\leq \norm{\chi_L G_K}_{\E}$, we have that $\norm{\chi_L G_K}_{X_T}$ is also uniformly bounded in $K$. Therefore, we have a subsequence converging weakly to some limit $\bar G_{L}\in X_T$ supported on the ball of radius $L-1$. Note that $\bar G_L$ and $\bar G_{L'}$ are equal on the ball of radius $\min(L,L')$ and for all $t \in [0,T]$.

Now we use a diagonal argument to take $L$ to infinity and extract a subsequence (still denoted $\{G_K\}$) and a limit $\bar G\in X_T$ such that,
$$G_K \rightharpoonup \bar G \text{ in }X_T \text { on compact sets}.$$

\lref{L_inf_a} and \lref{L_inf_c} imply that $\a[h]_\eps \part^2_{v_iv_j}G_K \rightharpoonup \a[h] \part^2_{v_iv_j}\bar G$, $\cm[h_\eps]G_K\rightharpoonup \cm[h]\bar G$,\\
$\a[h_\eps]\frac{v_i}{\jap{v}}\part_{v_j}G_K\rightharpoonup \a[h]\frac{v_i}{\jap{v}}\part_{v_j}\bar G$ and $\a[h_\eps]\frac{v_iv_j}{\jap{v}^2}G_K\rightharpoonup \a[h]\frac{v_iv_j}{\jap{v}^2}\bar G$ weakly in the space $H^2_{x,v}$ as $\eps\to 0$.\\
Indeed for $|\alpha|+|\beta|\leq 2$, and for any $\phi \in L^\infty([0,T];H^2_{x,v})$, we have,
\begin{align*}
&\norm{\der \phi[\der(\a[h]\part^2_{v_iv_j}\bar G-\a[h_\eps]\part^2_{v_iv_j}G_K)]}_{L^\infty([0,T];L^1_xL^1_v)}\\
&\leq \norm{\der \phi\der(\a[h-h_\eps]\part^2_{v_iv_j}G_K)}_{L^\infty([0,T];L^1_xL^1_v)}\\
&\quad+\norm{\der \phi\der(\a[h]\part^2_{v_iv_j}(\bar G-G_K))}_{L^\infty([0,T];L^1_xL^1_v)},\\
&\leq \norm{\phi}_{L^\infty([0,T];H^2_{x,v})}\norm{\h-\h_\eps}_{H^2_T}\norm{G_k}_{X_T}\\
&\quad+\norm{\h}_{L^\infty([0,T];\H)}\jap{\frac{\phi}{\jap{v}^{20}},\frac{\part^{2}_{v_iv_j}(\bar G-G_K)}{\jap{v}^{\frac{1}{2}}}}_{X^2_T}.
\end{align*}

Clearly $\frac{\phi}{\jap{v}^{20}}\in X^2_T$ and $\frac{\part^{2}_{v_iv_j}(\bar G-G_K)}{\jap{v}^{\frac{1}{2}}} \in X^8_{T}$ and since $G_K\rightharpoonup \bar G$ in $X_T$ the second term goes to zero. The first term goes to zero since $\norm{G_k}_{X_T}$ is uniformly bounded and $\h_\eps\to \h$ in $H^2_T$.\\
Similar computations show the result for other coefficients.

 Since we control at least $2$ derivatives for $\bar G$, it has sufficient regularity to be a solution to the linearized equation \eref{lin_eq_landau} in $W^{1,\infty}_tH^2_{x,v}$ on all of $\R^6$. Thanks to \lref{lin_visc_landau_exis} $G_K$ is nonnegative and thus $\bar G$ is too. Moreover, $\bar G$ inherits the bound \eref{lin_exp_bound} from \eref{gron}. \\   
\end{proof}
We are now finally in good shape to use contraction to construct a solution to \eref{eq_for_g}.
\begin{theorem}\label{t.soln_for_g}
Assume $g_{\ini}\in Y_{x,v}$ and that $$\norm{g_{\ini}}_{Y_{x,v}}\leq M_0.$$

Then for some $T \in [0,T_0)$ depending on $M_0$, there exists a nonnegative $g \in \E$ solving \eref{eq_for_g} with $g(0,x,v)=g_{\ini}(x,v)$.

Moreover, we have uniqueness of $g$ in $\tilde E_T^4\cap C^0([0,T);\tilde Y^4_{x,v})$.
\end{theorem}
\begin{proof}
Define $g^0(t,x,v)=g_{\ini}(x,v)$ and, for $n\geq 1$, define the sequence $\{g_n\}$ recursively as the solution of 
\begin{equation} \label{e.iter_eq_for_g}
\begin{split}
\partial_tg^n+v_i\partial_{x_i}g^n+\kappa\jap{v}g^n&=\a [f^{n-1}]\partial^2_{v_iv_j}g^n-\cm [f^{n-1}]g^n-2d(t)\a [f^{n-1}]\frac{v_i}{\jap{v}}\partial_{v_j} g^n
\\&\quad -d(t)\left(\frac{\delta_{ij}}{\jap{v}}-(d(t)+1)\frac{v_iv_j}{\jap{v}^2}\right)\a [f^{n-1}]g^n.
\end{split}
\end{equation}
Where $f^n=e^{-d(t)\jap{v}}g^n$ and $g^n(0,x,v)=g_{\ini}(x,v)$.

This is exactly the linearized equation \eref{lin_eq_landau}. Then by \lref{lin_eq}, for any $T\in[0,T_0)$, each $g^n$ exists, is nonnegative, belongs to the space $\E$ and satisfies,
\begin{equation}\label{e.iter_exp_bound}
\norm{g^n}^2_{\E}\leq \norm{g_{\ini}}^2_{E_{x,v}}\exp(C(d_0,\gamma)\norm{g^{n-1}}_{Y_T}T)
\end{equation}
for some $C(d_0,\gamma)>0$ that is independent of $n$.

Assume by induction that for $n\geq 1$,
\begin{equation}\label{e.iter_bound}
\norm{g^{n-1}}_{\H}\leq 2M_0
\end{equation}
for some $T\in (0,T_0]$. This hypothesis, holds for $n=1$ by our assumption on $g_{\ini}$. Then \eref{iter_exp_bound} becomes $$\norm{g^n}^2_{\E}\leq M_0^2\exp(2C(d_0,\gamma)M_0T)$$
If we take $$T\leq \min\left(\frac{2\ln 2}{C(d_0,\gamma)M_0},T_0\right),$$
then $\norm{g^n}_{\E}\leq 2M_0$. Also note that $T$ is independent of $n$. Thus, \eref{iter_bound} is true for all $n\geq 1$.

Now we define $w^n=g^n-g^{n-1}$. Equation \eref{iter_eq_for_g} implies for $n\geq 2$,
\begin{equation} \label{e.iter_eq_for_w}
\begin{split}
\partial_tw^n+v_i\partial_{x_i}w^n+\kappa\jap{v}w^n&=\a [f^{n-1}]\partial^2_{v_iv_j}w^n-\cm [f^{n-1}]w^n-2d(t)\a [f^{n-1}]\frac{v_i}{\jap{v}}\partial_{v_j} w^n
\\&\quad-d(t)\left(\frac{\delta_{ij}}{\jap{v}}-(d(t)+1)\frac{v_iv_j}{\jap{v}^2}\right)\a [f^{n-1}]w^n\\
&\quad+\a [v^{n-1}]\partial^2_{v_iv_j}g^{n-1}-\cm [v^{n-1}]g^{n-1}-2d(t)\a [v^{n-1}]\frac{v_i}{\jap{v}}\partial_{v_j} g^{n-1}
\\&\quad-d(t)\left(\frac{\delta_{ij}}{\jap{v}}-(d(t)+1)\frac{v_iv_j}{\jap{v}^2}\right)\a [v^{n-1}]g^{n-1},
\end{split}
\end{equation}
with $w^n(0,x,v)=0$ and $v^n=w^ne^{-d(t)\jap{v}}$.

For all multi-indices with $|\alpha|+|\beta|\leq 4$, we differentiate the equation for $w^n$ by $\der$, multiply by $\jap{v}^{2\tilde{\omega}_{\alpha,\beta}}\der w^n$, and integrate in space and velocity (note that we are doing the contraction in a much weaker space). 

Here $\tilde{\omega}_{\alpha,\beta}=10-(\frac{3}{2}+\delta_\gamma)|\alpha|-(\frac{1}{2}+\delta_\gamma)|\beta|$ and the spaces with tilde are the associated energy spaces defined in the same way as $\E^k$. The need to change the heirarchy for the contraction is due to the presence of terms with $g^{n-1}$ outside the coefficients because we can no longer do integration by parts.

For the first few terms where the coefficents depends on $f^{n-1}$, we use the estimates developed in the preceding sections and for the terms where we can't perform integration by parts, we just estimate $\part^\alpha\part^\beta g^{n-1}$ in $E_T$ and use that $\norm{g^n}_{\E}\leq 2M_0$ and hence  $\norm{w^n}_{\E}\leq 4M_0$.\\
More concretely, we have the estimate
$$\int_0^T\int_x\int_v \jap{v}^{2\tilde{\omega}_{\alpha,\beta}}\der w^n\derv{'}{'}\a[v^{n-1}]\part^{v_iv_j}\derv{''}{''}g^{n-1}\lesssim T\norm{w^n}_{\tilde{E}^4_T}\norm{w^{n-1}}_{\tilde{E}^4_T}\norm{g^{n-1}}_{E_T}.$$

This follows exactly in the same way as \lref{special_holder_der_geq_9} except we don't use Young's inequality to split the product and estimate everything in $L^\infty_t$ by putting all the extra $\jap{v}$ weights on $\part^2_{v_iv_j}\derv{''}{''}g^{n-1}$. This is precisely where fact that we are estimating in a weaker weighted space helps us.

The commutator term from $\kappa\der(\jap{v}w^n)$ gives a term of the form $C(d_0,\gamma)\kappa\int_0^T\norm{w^n}_{\tilde{Y}^4_{x,v}}$. All the other terms are absorbed on the left hand side by the term with the extra $\jap{v}$ weight.

We thus get,
\begin{equation*}
\begin{split}
\norm{w_n}^2_{\tilde{E}_T^4}&\leq C(d_0,\gamma)\kappa\int_0^T\norm{w^n}^2_{\tilde{Y}^4_{x,v}}(t)\d t\\
&\quad+C(d_0,\gamma)T\norm{w^n}_{\tilde{E}^4_T}\norm{w^{n-1}}_{\tilde{E}^4_T}\norm{g^{n-1}}_{E_T},\\
&\leq C(d_0,\gamma)T\kappa\norm{w^n}^2_{\tilde{Y}^4_{T}}+C(d_0,\gamma)T\norm{w^n}_{\tilde{E}^4_T}\norm{w^{n-1}}_{\tilde{E}^4_T}\norm{g^{n-1}}_{E_T}.
\end{split}
\end{equation*}

Now we use the fact that $\norm{g^{n-1}}_{E_T}\leq M_0$ and if necessary, choosing $T$ smaller, so that,
$$C(d_0,\gamma)\kappa T\leq \frac{1}{2}\text{  and  } C(d_0,\gamma)TM_0\leq \frac{1}{4}.$$

Thus after absorbing the first term on LHS we get,
\begin{equation}\label{e.cont}
\norm{g^n-g^{n-1}}_{\tilde{E}_T^4}^2\leq \frac{1}{2}\norm{g^{n-1}-g^{n-2}}_{\tilde{E}_T^4}\norm{g^n-g^{n-1}}_{\tilde{E}_T^4}.
\end{equation}

This implies that $\{g^n\}$ is a convergent sequence and $g\in \tilde{E}_T^4$ is a classical solution of \eref{eq_for_g}.

Now since $\{g_n\}$ is uniformly bounded in $\E$, it has subsequence that converges weakly to some $\bar g \in \E$. But then it is the weak limit in $\tilde{E}_T$ which implies that $g=\bar g$ a.e. Thus $g\in \E$.

Next we handle uniqueness in the same way: if $g_1$ and $g_2$ are two solutions of \eref{eq_for_g} in $\E$ with the same initial data, then $w:=g_1-g_2$ satisfies
\begin{equation*}
\begin{split}
\partial_tw+v_i\partial_{x_i}w+\kappa\jap{v}w&=\a [g_2]\partial^2_{v_iv_j}w-\cm [g_2]w-2d(t)\a [g_2]\frac{v_i}{\jap{v}}\partial_{v_j} w\\
&\quad-d(t)\left(\frac{\delta_{ij}}{\jap{v}}-(d(t)+1)\frac{v_iv_j}{\jap{v}^2}\right)\a [g_2]w\\
&\quad+\a[w]\partial^2_{v_iv_j}g_2-\cm [w]g_2-2d(t)\a [w]\frac{v_i}{\jap{v}}\partial_{v_j} g_2\\
&\quad-d(t)\left(\frac{\delta_{ij}}{\jap{v}}-(d(t)+1)\frac{v_iv_j}{\jap{v}^2}\right)\a [w]g_2,
\end{split}
\end{equation*}
and $w(0,x,v)=0$. By the same estimate as above, Gr\"{o}nwall's inequality, we conclude that $\norm{w}_{\tilde{E}_T^4}=0$.
\end{proof}

Now \tref{soln_for_g} implies the main result, \tref{local} with $f=e^{-(d_0-\kappa t)\jap{v}}g$.

\end{document}